\documentclass[a4paper,reqno]{amsart}

\usepackage[english]{babel}
\usepackage[all]{xy}
\usepackage{amssymb}
\usepackage{verbatim}
\usepackage{geometry}
\usepackage{enumerate}
\usepackage{hyperref}

\DeclareMathOperator{\End}  {End}
\DeclareMathOperator{\diag} {diag}
\DeclareMathOperator{\Id}   {Id}
\DeclareMathOperator{\Pic}  {Pic}
\DeclareMathOperator{\Supp} {Supp}
\DeclareMathOperator{\Proj} {Proj}

\theoremstyle{plain}
\newtheorem*{teo1}{Theorem 1}
\newtheorem*{teo2}{Theorem 2}
\newtheorem{lem}{Lemma}[section]
\newtheorem{teo}[lem]{Theorem}
\newtheorem{prop}[lem]{Proposition}
\newtheorem{cor}[lem]{Corollary}

\theoremstyle{definition}
\newtheorem*{dfn*}{Definition}
\newtheorem{dfn}[lem]{Definition}

\theoremstyle{remark}
\newtheorem{oss}[lem]{Remark}

\let\overlinedmarginpar\marginpar
\renewcommand\marginpar[1]{\-\overlinedmarginpar[\raggedleft\footnotesize
    #1]{\raggedright\footnotesize #1}}

\begin{document}

\title[Simple linear compactifications of odd orthogonal groups]{Simple linear compactifications of\\odd orthogonal groups}

\author{Jacopo Gandini}

\date{\today}

\curraddr{\textsc{Dipartimento di Matematica ``Guido Castelnuovo''\\
                    ``Sapienza'' Universit\`a di Roma\\
                    Piazzale Aldo Moro 5\\
                    00185 Roma, Italy}}

\email{gandini@mat.uniroma1.it}

\begin{abstract}
We classify the simple linear compactifications of $\mathrm{SO}(2r+1)$, namely those compactifications with a unique closed orbit which are obtained by taking the closure of the $\mathrm{SO}(2r+1)\times \mathrm{SO}(2r+1)$-orbit of the identity in a projective space $\mathbb{P}(\End(V))$, where $V$ is a finite dimensional rational $\mathrm{SO}(2r+1)$-module.
\end{abstract}

\maketitle

\section*{Introduction}

Let $G$ be a semisimple and simply-connected algebraic group defined over an algebraically closed field $\Bbbk$ of characteristic zero. Fix a maximal torus $T \subset G$ and a Borel subgroup $B \supset T$, denote $\Phi$ the associated root system of $G$ and $\Delta \subset \Phi$ the associated basis. Denote $\Lambda$ the weight lattice of $G$ and $\Lambda^+ \subset \Lambda$ the semigroup of dominant weights. For $\lambda \in \Lambda^+$, denote $V(\lambda)$ the simple $G$-module of highest weight $\lambda$. 

A $G\times G$-variety $X$ is called \textit{linear} if it admits an equivariant embedding in the projective space of a finite dimensional $G\times G$-module, while is called \textit{simple} if it possesses a unique closed $G\times G$-orbit. If $\Pi\subset \Lambda^+$ is a finite subset, consider the $G\times G$-variety
\[	X_\Pi = \overline{(G\times G)[\Id]} \subset \mathbb{P}\Big(\bigoplus_{\lambda \in \Pi} \End(V(\lambda)\Big):	\]
it is a linear compactification of a quotient of $G$, and conversely every linear compactification of a quotient of $G$ arise in such a way for some $\Pi \subset \Lambda^+$. We say that $X_\Pi$ is \textit{adjoint} if it is a compactification of a quotient of the adjoint group $G_\mathrm{ad}$.

The variety $X_\Pi$ was studied by Timashev in \cite{Ti}: there are studied the local structure and the $G\times G$-orbit structure, and normality and smoothness are characterized as well. The conditions of normality in particular rely on some properties of the tensor product, and together with the conditions of smoothness they were remarkably simplified by Bravi, Gandini, Maffei, Ruzzi in \cite{BGMR} in case $X_\Pi$ is simple and adjoint, and by Gandini, Ruzzi in \cite{GR} in case $X_\Pi$ is simple. In particular, in \cite{BGMR} it was shown that every simple adjoint linear compactification is normal if $G$ is simply laced, whereas several examples of non-normal simple adjoint linear compactifications arise in the non-simply laced case.

By a theorem of Sumihiro (see \cite{KKLV}), every simple normal $G\times G$-variety is linear. Hence if we restrict to simple normal adjoint $X_\Pi$'s, a classification follows by the general Luna-Vust theory of spherical embeddings (see \cite{Kn}): they are classified by their closed orbits, i.e. by non-empty subsets of $\Delta$. However, as far as we know, no explicit classification is known in the general spherical context without assuming normality: this paper stems from the attempt to understand this classification in some explicit case. More precisely, the aim of this work is to classify the simple linear compactifications of $\mathrm{SO}(2r+1)$: this will be done by classifying the subsets $\Pi \subset \Lambda^+$ which give rise to isomorphic simple compactifications.

Consider the \textit{dominance order} on $\Lambda$, defined by $\mu \leqslant \lambda$ if $\lambda - \mu \in \mathbb{N} \Delta$, and the \textit{rational dominance order} $\leqslant_\mathbb{Q}$, defined by $\mu \leqslant_\mathbb{Q} \lambda$ if $\lambda - \mu \in \mathbb{Q}_{\geqslant 0} \Delta$. If $\Pi \subset \Lambda^+$ is finite, then the closed orbits of $X_\Pi$ correspond to some maximal elements of $\Pi$ w.r.t. $\leqslant_\mathbb{Q}$, and $X_\Pi$ is simple if and only if $\Pi$ contains a unique maximal element w.r.t. $\leqslant_\mathbb{Q}$. If this is the case, we say that $\Pi \subset \Lambda^+$ is a \textit{simple subset}. On the other hand, $X_\Pi$ is an adjoint compactification if and only if $\Pi$ is contained in a coset of $\Lambda/\mathbb{Z}\Delta$, in which case we say that $\Pi \subset \Lambda^+$ is an \textit{adjoint subset}. Therefore $X_\Pi$ is a simple adjoint variety if and only if $\Pi$ contains a unique maximal element w.r.t. $\leqslant$.

For simplicity, in case $\Pi = \{\lambda\}$, we denote $X_\Pi$ by $X_\lambda$, while in case $\Pi = \{\lambda, \mu\}$, we denote $X_\Pi$ by $X_{\lambda, \mu}$. Let $\Pi\subset \Lambda^+$ be a simple adjoint subset with maximal element $\lambda$, denote $\widetilde{X}_\lambda$ the normalization of $X_\lambda$ and $\Pi^+(\lambda) = \{ \mu \in \Lambda^+ \, : \, \mu \leqslant \lambda \}$. Then $\Pi \subset \Pi^+(\lambda)$ and we get natural projections
\[	X_{\Pi^+(\lambda)} \longrightarrow X_\Pi \longrightarrow X_\lambda	\]
While Kannan shown in \cite{Ka} that $X_{\Pi^+(\lambda)}$ is projectively normal, De~Concini proved in \cite{DC} that $X_{\Pi^+(\lambda)} \simeq \widetilde{X}_\lambda$. In particular, if $\Pi$ is adjoint and simple with maximal element $\lambda$, it follows that $\widetilde{X}_\lambda \to X_\Pi$ is the normalization.

If $\lambda\in \Lambda^+$, we say that a weight $\mu \in \Pi^+(\lambda)$ is \textit{trivial} if $X_{\lambda, \mu}$ is equivariantly isomorphic to $X_\lambda$. We denote by $\Pi^+_\mathrm{tr}(\lambda) \subset \Pi^+(\lambda)$ the subset of trivial weights: if $G$ is simply laced, then by \cite{BGMR} we have $\Pi^+_\mathrm{tr}(\lambda) = \Pi^+(\lambda)$. If $\Supp(\lambda)\subset \Delta$ denotes the set of simple roots non-orthogonal to $\lambda$, then the variety $X_\lambda$ depends only on $\Supp(\lambda)$: therefore a first step to classify the simple linear compactifications $X_\Pi$ such that $\widetilde{X}_\lambda \to X_\Pi \to X_\lambda$ is to characterize the set $\Pi^+_\mathrm{tr}(n\lambda)$ for $n \in \mathbb{N}$.

In the case $G = \mathrm{Spin}(2r+1)$, we will give the following combinatorial description of trivial weights. Denote $\Delta = \{\alpha_1, \ldots, \alpha_r\}$, where the numbering is the usual one as in \cite{Bo}, and denote $\omega_1, \ldots, \omega_r$ the associated fundamental weights.

\begin{teo1}[see Theorem \ref{teo II: classificazione pesi banali B}]
Let $G = \mathrm{Spin}(2r+1)$. Let $\lambda \in \Lambda^+$ and $\mu \in \Pi^+(\lambda)$, denote $q$ and $l$ the maximal integers such that $\alpha_q \in \Supp(\lambda)$ and $\alpha_l \in \Supp(\mu)$ and write $\lambda-\mu = \sum_{i=1}^r a_i \alpha_i$. Then $\mu \in \Pi^+_\mathrm{tr}(\lambda)$ if and only if $a_r$ is even or $a_r > 2\min\{r-l,r-q\}$.
\end{teo1}

Previous theorem essentially expresses some properties of the tensor product. A main motivation to explain the combinatorial condition in the previous theorem arises by considering the case of the first fundamental weight, where it can be deduced by the Schur-Weyl duality (Proposition \ref{prop II: schur-weyl1}).

To reduce the classification of the simple linear compactifications of $\mathrm{SO}(2r+1)$ to the classification of the trivial weights, it is possible to define a partial order relation $\leqslant^\lambda$ on $\Lambda$ with the following geometrical meaning: if $\mu, \nu \in \Pi^+(\lambda) \smallsetminus \Pi^+_\mathrm{tr}(\lambda)$, then
\[	\nu \leqslant^\lambda \mu 	\qquad 	\text{ if and only if }	\qquad
	\begin{array}{c}
		\text{ there exists a $G\times G$-morphism}	\\
		X_{\lambda,\mu} \longrightarrow X_{\lambda,\nu}	
	\end{array}	\]
From a combinatorial point of view, $\leqslant^\lambda$ is the degeneration of the dominance order associated to the set $\Phi^+(\lambda)$ of the positive roots of $\Phi$ which are non-orthogonal to $\lambda$: if $\mu, \nu \in \Lambda$, then
$\nu \leqslant^\lambda \mu$ if and only if $\mu - \nu \in \mathbb{N} \Phi^+(\lambda)$.

In case $\lambda$ is regular, then $\leqslant^\lambda$ coincides with the usual dominance order $\leqslant$, while if $\lambda = 0$ then $\leqslant^\lambda$ is the trivial order. In the general case of a (possibly non-adjoint) simple subset $\Pi$, the partial order $\leqslant^\lambda$ was used in \cite{GR} to characterize combinatorially the normality of the variety $X_\Pi$.

If $\Pi \subset \Lambda^+$ is an adjoint simple subset with maximal element $\lambda$, denote
\[	\Pi_{\mathrm{red}} = \{ \mu \in \Pi \smallsetminus \Pi^+_\mathrm{tr}(\lambda) \, : \,  \mu \text{ is maximal w.r.t. } \leqslant^\lambda \} \cup \{\lambda\}.	\]
In case $\Pi = \Pi_{\mathrm{red}}$ we say that $\Pi$ is a \textit{reduced adjoint subset}. If $\Pi'$ is another adjoint simple subset with maximal element $\lambda'$, then we say that $\Pi$ and $\Pi'$ are \textit{equivalent} (and we write $\Pi \sim \Pi'$) if $\Supp(\lambda') = \Supp(\lambda)$ and $\Pi' - \lambda' = \Pi - \lambda$.

\begin{teo2} [see Corollary \ref{cor III: morfismi gen}]
Let $G = \mathrm{Spin(2r+1)}$ and let $\Pi, \Pi'\subset \Lambda^+$ be adjoint simple subsets with maximal elements resp. $\lambda$ and $\lambda'$.
\begin{itemize}
		\item[i)] Suppose that $\Supp(\lambda) = \Supp(\lambda')$. There exists an equivariant morphism $X_\Pi \to X_{\Pi'}$ if and only if for every $\mu' \in \Pi'_\mathrm{red}$ there exists $\mu \in \Pi_\mathrm{red}$ such that $\mu' - \lambda' \leqslant^\lambda \mu - \lambda$.
		\item[ii)] The varieties $X_\Pi$ and $X_{\Pi'}$ are equivariantly isomorphic if and only if $\Pi_\mathrm{red} \sim \Pi'_\mathrm{red}$.
\end{itemize}
\end{teo2}

As a corollary it follows that the simple linear compactifications of $\mathrm{SO}(2r+1)$ are classified by simple reduced subsets up to equivalence.

The paper is organized as follows. In Section 1, we study the compactifications $X_\lambda$ and $X_\Pi$ in full generality: throughout this section (and only in this section) $G$ will denote an arbitrary simply connected semisimple algebraic group. In Section 2, we describe combinatorially the set of trivial weights $\Pi^+_{\mathrm{tr}}(\lambda)$, where $\lambda$ is a dominant weight for $\mathrm{Spin}(2r+1)$, and we prove Theorem 1. In Section 3, we introduce the reduction of a simple subset and we characterize combinatorially the existence of an equivariant morphism between two simple linear compactifications of $\mathrm{SO}(2r+1)$ possessing isomorphic closed orbits and we prove Theorem 2. In Section 4, we give examples by means of tables in the case of the simple linear compactifications of $\mathrm{SO}(7)$ and of $\mathrm{SO}(9)$.

Differently from the introduction, since we will only deal with adjoint compactifications, we will refer to \textit{simple adjoint sets} of dominant weights just as \textit{simple sets}. It will be also convenient to adopt a ``dual viewpoint'' in the definition of the variety $X_\Pi$: the simple modules $V(\mu)$ occurring in its definition will be substituted with their duals.\\

\textit{Acknowledgments.} I would like to thank A.~Maffei, who suggested the problem, for his advice and for his continuous support. As well, I thank the referee for his careful reading and useful comments.

\section{The varieties $X_\lambda$ and $X_\Pi$}

Let $G$ be a semisimple simply connected algebraic group over an algebraically closed field $\Bbbk$ of characteristic zero and denote $\mathfrak{g}$ its Lie algebra. Fix a maximal torus $T$ and a Borel subgroup $B\supset T$, denote $B^-$ the opposite Borel subgroup of $B$ w.r.t. $T$ and denote $U \subset B$ and $U^- \subset B^-$ the associated maximal unipotent subgroups. Correspondingly to the choice of $T$ and $B$, we fix $T\times T$ as a maximal torus and $B\times B^-$ as a Borel subgroup in $G\times G$. Denote $\Phi$ the root system associated to $T$ and $W$ the Weyl group of $\Phi$. Denote $\Delta$ the basis of $\Phi$ associated to $B$ and $\Phi^+$ the associated set of positive roots. In case $\Phi$ is irreducible and non-simply laced, then we write $\Phi^+ = \Phi^+_s \cup \Phi^+_l$, where $\Phi^+_s$ and $\Phi^+_l$ denote respectively the set of the positive short roots and that of the positive long roots. 

Denote $\Lambda$ the weight lattice of $\Phi$ and $\Lambda^+ \subset \Lambda$ the semigroup of the dominant weights associated to $\Delta$ and set $\Lambda_\mathbb{Q} = \Lambda \otimes \mathbb{Q}$. If $\lambda\in \Lambda^+$ then we denote by $V(\lambda)$ the simple $G$-module of highest weight $\lambda$, however, if we deal with different groups, we will use also the notation $V_G(\lambda)$. Let $\lambda \mapsto \lambda^*$ be the linear involution of $\Lambda$ defined by $V(\lambda)^* \simeq V(\lambda^*)$ for any dominant weight $\lambda$. If $\alpha\in \Delta$, we denote by $\omega_\alpha$ the associated fundamental weight and by $\{e_\alpha, \alpha^{\vee},f_\alpha\}$ an $\mathfrak{sl}(2)$-triple of $T$-weights $\alpha,0,-\alpha$. Given a weight $\lambda = \sum_{\alpha \in \Delta} n_\alpha \omega_\alpha \in \Lambda$, denote $\lambda^+ = \sum_{n_\alpha > 0} n_\alpha \omega_\alpha$ and $\lambda^- = \sum_{n_\alpha < 0} |n_\alpha| \omega_\alpha$. When we deal with explicit root systems, we use the numbering of simple roots and fundamental weights of Bourbaki \cite{Bo}.

If $\lambda\in \Lambda^+$, denote $\Pi(\lambda) \subset \Lambda$ the set of weights occurring in $V(\lambda)$, $\mathcal{P}(\lambda) \subset \Lambda_\mathbb{Q}$ the convex hull of $\Pi(\lambda)$ and $\Pi_\mathbb{Q}(\lambda) = \mathcal{P}(\lambda) \cap \Lambda$. Denote $\leqslant$ and $\leqslant_\mathbb{Q}$ resp. the \textit{dominance order} and the \textit{rational dominance order} on $\Lambda$, defined by $\mu \leqslant \lambda$ (resp. $\mu \leqslant_\mathbb{Q} \lambda$) if and only if $\lambda - \mu \in \mathbb{N} \Delta$ (resp. $\mathbb{Q}_{\geqslant 0} \Delta$). Then we have $\Pi(\lambda) = W \Pi^+(\lambda)$ and $\Pi_\mathbb{Q}(\lambda) = W\Pi_\mathbb{Q}^+(\lambda)$, where we denote
\begin{align*}
	\Pi^+(\lambda) &= \Pi(\lambda) \cap \Lambda^+ = \{ \mu \in \Lambda^+ \, : \, \mu \leqslant \lambda \}, \\	\Pi_\mathbb{Q}^+(\lambda) &= \mathcal{P}(\lambda) \cap \Lambda^+ = \{ \mu \in \Lambda^+ \, : \, \mu \leqslant_\mathbb{Q} \lambda \}.
\end{align*}

If $\lambda \in \Lambda$, define its \emph{support} as $\Supp(\lambda) = \{\alpha \in \Delta \, : \, \langle \lambda, \alpha^\vee \rangle \neq 0 \}$, while if $\theta = \sum_{\alpha \in \Delta} n_\alpha \alpha \in \mathbb{Z} \Delta$, define its \textit{support over} $\Delta$ as $\Supp_\Delta(\theta) = \{ \alpha \in \Delta \, : \, n_\alpha \neq 0 \}$. If $\lambda \in \Lambda$, denote $\Phi^+(\lambda) \subset \Phi^+$ the subset of the positive roots which are non-orthogonal to $\lambda$:
\[	\Phi^+(\lambda) = \{ \theta \in \Phi^+ \, : \, \Supp_\Delta(\theta) \cap \Supp(\lambda) \neq \varnothing \}.	\]
In case $\Phi$ is irreducible and non-simply laced, we also set $\Phi^+_s(\lambda) = \Phi^+_s \cap \Phi^+(\lambda)$ and $\Phi^+_l(\lambda) = \Phi^+_l \cap \Phi^+(\lambda)$.

If $\lambda$ is a non-zero dominant weight, consider the $G\times G$-variety
\[ X_\lambda = \overline{(G\times G)[\Id]}\subset \mathbb{P}(\End(V(\lambda)^*)), \]
Since $\mathbb{P}(\End(V(\lambda)^*))$ possesses a unique closed $G\times G$ orbit and since the diagonal of $G$ fixes the identity, it follows that $X_\lambda$ is a simple compactification of a quotient of the adjoint group $G_\mathrm{ad}$.

\begin{prop} [{\cite[Prop.~1.2]{BGMR}}] \label{prop I: morfismi X_lambda}
Let $\lambda, \mu \in \Lambda^+$. Then $X_\lambda \simeq X_\mu$ as $G\times G$-varieties if and only if $\Supp(\lambda) = \Supp(\mu)$.
\end{prop}

If $\lambda \in \Lambda^+$ is regular (i.e. if $\Supp(\lambda) = \Delta$), then $X_\lambda$ coincides with the wonderful compactification of $G_\mathrm{ad}$ introduced by De~Concini and Procesi in \cite{CP}. We will denote this variety by $M$: it is smooth and the complement of its open orbit is the union of smooth prime divisors with normal crossings whose intersection is the closed orbit $G/B \times G/B$.

Since $G$ is semisimple and simply connected, we may identify the Picard group $\Pic(G/B)$ with the weight lattice $\Lambda$: we identify a weight $\lambda\in \Lambda$ with the line bundle on $G/B$ whose $T$-weight in the $B$-fixed point is $-\lambda$. The restriction of line bundles to the closed orbit induces an homomorphism
\[	\omega : \Pic(M) \longrightarrow \Lambda \times \Lambda	\]
which is injective and which identifies $\Pic(M)$ with the sublattice $\{ (\lambda,\lambda^*) \, : \, \lambda \in \Lambda\} \subset \Lambda \times \Lambda$. Therefore $\Pic(M)$ is identified with $\Lambda$ and we will still denote by
$\mathcal{L}_\lambda \in \Pic(M)$ the line bundle whose image is $(\lambda, \lambda^*)$. A line bundle $\mathcal{L}_\lambda \in \Pic(M)$ is generated by its sections if and only if $\lambda\in \Lambda^+$, in which case, as a $G\times G$-module, the following decomposition holds (\cite[Theorem 8.3]{CP}):
\[	\Gamma(M,\mathcal{L}_\lambda) \simeq \bigoplus_{\mu \in \Pi^+(\lambda)} \End(V(\mu)).	\]

Fix now $\lambda \in \Lambda^+$ (possibly non-regular). Then the map $G_\mathrm{ad} \to \mathbb{P}(\End(V(\lambda)^*))$ extends to a map $M \to \mathbb{P}(\End(V(\lambda)^*))$ whose image is $X_\lambda$ and such that $\mathcal{L}_\lambda$ is the pullback of the hyperplane bundle on $\mathbb{P}(\End(V(\lambda)^*))$. If we pull back the homogeneous coordinates of $\mathbb{P}(\End(V(\lambda)^*))$ to $M$, we get a submodule of $\Gamma(M,\mathcal{L}_\lambda)$ which is isomorphic to $\End(V(\lambda))$; by abuse of notation we still denote this submodule by $\End(V(\lambda))$.

Consider the algebra
\[	\widetilde{A}(\lambda) = \bigoplus_{n\in \mathbb{N}} \Gamma(M, \mathcal{L}_{n\lambda})	\]
and denote $A(\lambda) \subset \widetilde{A}(\lambda)$ the subalgebra generated by $\End(V(\lambda)) \subset \Gamma(M,\mathcal{L}_\lambda)$; consider the natural gradings on $\widetilde{A}(\lambda)$ and $A(\lambda)$ respectively defined by $\widetilde{A}_n(\lambda) = \Gamma(M, \mathcal{L}_{n\lambda})$ and $A_n(\lambda) = \widetilde{A}_n(\lambda)\cap A(\lambda)$. Since $A(\lambda)$ is the projective coordinate ring of $X_\lambda$, if we set $\widetilde{X}_\lambda= \Proj \widetilde{A}(\lambda)$ then we get a commutative diagram as follows:
\[	\xymatrix{M \ar@{->>}[rr] \ar@{->>}[drr] & & \widetilde{X}_\lambda \ar@{->>}[d] \\	& & X_\lambda}	\]
In \cite{Ka}, Kannan shown that $\widetilde{A}(\lambda)$ is generated in degree $1$, while in \cite{DC} De~Concini shown that $\widetilde{X}_\lambda \longrightarrow X_\lambda$ is the normalization.

If $\lambda \in \Lambda^+$, denote $\Id_\lambda \in \End(V(\lambda)^*)$ the identity. Similarly, if $\Pi \subset \Lambda^+$ is a finite subset, denote $\mathrm{Id}_\Pi$ the identity vector $(\mathrm{Id}_\mu)_{\mu \in \Pi} \in \bigoplus_{\mu \in \Pi} \End(V(\mu))$. Given such a subset $\Pi$, consider the $G\times G$-variety
\[
	X_\Pi = \overline{(G\times G)[\Id_\Pi]} \subset \mathbb{P}\Big(\bigoplus_{\mu \in \Pi} \End(V(\mu)^*)\Big).
\]
If $\Pi = \{\lambda\}$, then we get the variety $X_\lambda$, while if $\Pi = \Pi^+(\lambda)$ we get its normalization $\widetilde{X}_\lambda$. If $\Pi = \{\mu_1, \ldots, \mu_m\}$, for simplicity we will denote $X_\Pi$ also by $X_{\mu_1, \ldots, \mu_m}$. We say that the variety $X_\Pi$ is \textit{simple} if it contains a unique closed $G\times G$ orbit, while we say that $X_\Pi$ is \textit{adjoint} if it is a compactification of a quotient of $G_\mathrm{ad}$.

\begin{prop} [{\cite[\S 8]{Ti}}]	\label{prop I: timashev}
Let $\Pi \subset \Lambda^+$ be a finite subset and denote $\mathcal P(\Pi) \subset \mathbb{Q}\Delta$ the polytope generated by the $T$-weights occurring in the $G$-module $\oplus_{\mu \in \Pi}V(\mu)$.
\begin{itemize}
	\item[i)] $X_\Pi$ is adjoint if and only if $\Pi$ is contained in a coset of $\Lambda /\mathbb{Z}\Delta$.
	\item[ii)] Let $\mu \in \Pi$. Then $X_\Pi$ contains the closed orbit of $\mathbb{P}(\End(V(\mu)^*)$ if and only $\mu$ is an extremal vertex of $\mathcal P(\Pi)$.
\end{itemize}
\end{prop}

It follows by previous proposition that $X_\Pi$ is simple if and only if $\Pi$ possesses a unique maximal element w.r.t. $\leqslant_\mathbb{Q}$, whereas it is simple and adjoint if and only if $\Pi$ contains a unique maximal element w.r.t. $\leqslant$. Correspondingly, we will say that $\Pi$ is \textit{adjoint} if it is contained in a coset of $\Lambda/\mathbb{Z}\Delta$, and we say that an adjoint subset is \textit{simple} if it possesses a unique maximal element w.r.t. $\leqslant$. Since we will deal only with adjoint subsets, for simplicity we will refer to adjoint simple subsets just as \textit{simple subsets}. For a general treatment on the case of a possibly non-adjoint simple linear group compactification see \cite{GR}.

Suppose that $\Pi \subset \Lambda^+$ is simple with maximal element $\lambda$ and consider the line bundle $\mathcal{L}_\lambda \in \Pic(M)$. By its decomposition, it follows that $\Gamma(M,\mathcal{L}_\lambda)$ possesses a $G\times G$ submodule $A_1(\Pi)$ isomorphic to $\bigoplus_{\mu \in \Pi} \End(V(\mu))$, which is base point free since $\lambda \in \Pi$. On the other hand $\Gamma(M,\mathcal L_\lambda) = A_1(\Pi^+(\lambda))$ and $X_{\Pi^+(\lambda)} \simeq \widetilde X_\lambda$, hence we get morphisms
\[	M \longrightarrow \widetilde X_\lambda \longrightarrow X_\Pi \longrightarrow X_\lambda	\]
and it follows that $\widetilde X_\lambda \longrightarrow X_\Pi$ is the normalization.  Denote $A(\Pi)$ the projective coordinate ring of $X_\Pi$, namely the subalgebra of $\widetilde{A}(\lambda)$ generated by $\bigoplus_{\mu \in \Pi} \End(V(\mu))$, and denote $A_n(\Pi) = \widetilde{A}_n(\lambda) \cap A(\Pi)$.

If $\mu \in \Pi$, denote $\phi_\mu \in \End(V(\mu))$ a highest weight vector. Consider the $B\times B^-$-stable affine open subsets $X_\lambda^\circ \subset X_\lambda$ and $X_\Pi^\circ \subset X_\Pi$ defined by the non-vanishing of $\phi_\lambda$: then we get
\[
	\Bbbk[X_\Pi^\circ] = \left\{ \frac{\phi}{\phi_\lambda^n} \, : \, \phi \in A_n(\Pi), n \in \mathbb{N} \right\} \supset \left\{ \frac{\phi}{\phi_\lambda^n} \, : \, \phi \in A_n(\lambda), n \in \mathbb{N}  \right\} = \Bbbk[X_\lambda^\circ].	\]
Previous rings are not $G\times G$-module, however they are $\mathfrak{g}\oplus \mathfrak{g}$-modules.

\begin{lem} \label{lem I: generatori X-lambda-mu}
Let $\Pi \subset \Lambda^+$ be simple with maximal element $\lambda$. As a $\mathfrak{g} \oplus \mathfrak{g}$-algebra, $\Bbbk[X^\circ_\Pi]$ is generated by $\Bbbk[X^\circ_{\lambda}]$ together with the set $\{\phi_\mu/\phi_\lambda\}_{\mu \in \Pi}$.
\end{lem}

\begin{proof}
Since the projective coordinate ring $A(\Pi)$ is generated by its degree one component $A_1(\Pi) = \bigoplus_{\mu\in \Pi} \End(V(\mu))$, it follows that $\Bbbk[X^\circ_\Pi]$ is generated as an algebra by its subset $B(\Pi) = \left\{\phi/\phi_\lambda \, : \, \phi \in A_1(\Pi) \right\}$. Using the action of $\mathfrak{g} \oplus \mathfrak{g}$, let's show that $B(\Pi)$ is contained in the $\mathfrak{g} \oplus \mathfrak{g}$-subalgebra $B'(\Pi)\subset \Bbbk[X^\circ_\Pi]$ generated by $\Bbbk[X^\circ_{\lambda}]$ together with $\{\phi_\mu/\phi_\lambda\}_{\mu \in \Pi}$. Suppose indeed that $\alpha$ is a simple root and that $\phi/\phi_\lambda \in B'(\Pi)$: then $f_\alpha(\phi)/\phi_\lambda \in B'(\Pi)$ as well since 
\[	\frac{f_\alpha(\phi)}{\phi_\lambda} = f_\alpha\left(\frac{\phi}{\phi_\lambda}\right) + \frac{\phi}{\phi_\lambda} \cdot \frac{f_\alpha(\phi_\lambda)}{\phi_\lambda}. \qedhere	\]
\end{proof}

Given $\lambda, \mu \in \Lambda^+$, consider the multiplication map
\[	m_{\lambda, \mu} : \Gamma(M,\mathcal{L}_\lambda) \times \Gamma(M,\mathcal{L}_\mu) \longrightarrow \Gamma(M,\mathcal{L}_{\lambda+\mu}),	\]
which is surjective by \cite{Ka}. In order to describe combinatorially $m_{\lambda, \mu}$, as in \cite{Ka} or in \cite{DC} it is possible to identify sections of a line bundle on $M$ with functions on $G$ and use the description of the multiplication of matrix coefficients. Recall that as a $G\times G$-module it holds the decomposition
\[	\Bbbk[G]=\bigoplus_{\lambda\in \Lambda^+}\End(V(\lambda)) \simeq \bigoplus_{\lambda\in \Lambda^+} V(\lambda)^*\otimes V(\lambda).	\]
More explicitly if $V$ is a $G$-module, define the matrix coefficient $c_V:V^*\otimes V \to \Bbbk[G]$ by $c_V(\psi \otimes v)(g)= \langle \psi, gv\rangle$. If we multiply functions in $\Bbbk[G]$ of this type then we get
\[	c_{V}( \psi \otimes v) \cdot c_{W}(\chi \otimes w) = c_{V\otimes W} \big( (\psi\otimes\chi) \otimes (v\otimes w) \big):	\]
in particular we get that the image of the multiplication $\End(V(\lambda)) \otimes \End(V(\mu)) \to \Bbbk[G]$ is the sum of all $\End(V(\nu))$ with $V(\nu)\subset V(\lambda)\otimes V(\mu)$. We get then the following combinatorial description of $m_{\lambda, \mu}$.

\begin{lem}[{\cite[Lemma~3.4]{DC}, \cite[Lemma~3.1]{Ka}}] \label{lem I: coefficientimatriciali}
Let $\lambda, \mu \in \Lambda^+$, $\lambda' \in \Pi^+(\lambda)$ and $\mu' \in \Pi^+(\mu)$. Then the image of $\End(V(\lambda')) \times \End(V(\mu')) \subset \Gamma(M, \mathcal{L}_{\lambda}) \times \Gamma(M, \mathcal{L}_{\mu})$ in $\Gamma(M, \mathcal{L}_{\lambda + \mu})$ via $m_{\lambda,\mu}$ is
\[	\bigoplus_{V(\nu)\subset V(\lambda')\otimes V(\mu')} \End(V(\nu))	\]
\end{lem}

If $\Pi \subset \Lambda^+$ is simple with maximal element $\lambda$, define
\[	\Omega(\Pi) = \Big\{ \nu - n\lambda \, : \, V(\nu) \subset \Big(\bigoplus_{\mu \in \Pi} V(\mu)\Big)^{\otimes n} \Big\}.	\]
Notice that if $\nu \in \Pi$ then $\nu - \lambda \in \Omega(\Pi)$. Notice also that, if $\Pi_1 \subset \Pi_2$ are simple subsets with the same maximal element, then $\Omega(\Pi_1) \subset \Omega(\Pi_2)$. If $\Pi = \{\mu_1, \ldots, \mu_m\}$, for simplicity sometimes we will denote $\Omega(\Pi)$ also by $\Omega(\mu_1, \ldots, \mu_m)$.

\begin{oss}	\label{oss I: Omega(Pi)}
If $V_\Pi \doteq \bigoplus_{\mu \in \Pi} V(\mu)$, then by Lemma \ref{lem I: coefficientimatriciali} we have $A_n(\Pi) = \bigoplus_{V(\nu)\subset V_\Pi^{\otimes n}} \End(V(\nu))$. Hence by the description of $\Bbbk[X^\circ_\Pi]$ we get the following identification
\[	\Omega(\Pi) \simeq \left\{(\theta, \theta^*) \, : \, \theta \in \Omega(\Pi)\right\} = \left\{B\times B^-\text{-weights in } \Bbbk[X^\circ_\Pi] \right\},	\]
and in particular $\Omega(\Pi)$ is a semigroup of $\mathbb Z \Delta$ respect to the addition.
\end{oss}

Since $BB^-\subset G$ is an open subset, it follows that $G$ is a spherical $G\times G$ variety and every non-zero $B\times B^-$-semiinvariant function $\phi \in \Bbbk(G)$ is uniquely determined by its weight up to a scalar factor. Following the general theory of spherical varieties (see for instance \cite{Kn}), the semigroup $\Omega(\Pi)$ encodes a lot of information on the geometry of $X_\Pi$. In particular, we may characterize the existence of an equivariant morphism $X_\Pi \to X_{\Pi'}$ in terms of the semigroups $\Omega(\Pi)$ and $\Omega(\Pi')$ and of the isotypic decomposition of the tensor powers of $\bigoplus_{\mu \in \Pi} V(\mu)$.

\begin{prop} \label{prop I: criterio tensoriale}
Let $\Pi, \Pi'$ be simple subsets with resp. maximal elements $\lambda,\lambda'$ and suppose that $\Supp(\lambda) = \Supp(\lambda')$. There exists a $G \times G$-equivariant morphism $X_\Pi \to X_{\Pi'}$ if and only if, for every $\mu' \in \Pi'$, they exist $\mu_1, \ldots, \mu_n \in \Pi$ such that $V(\mu' - \lambda' + n \lambda ) \subset V(\mu_1) \otimes \ldots \otimes V(\mu_n)$.
\end{prop}

\begin{proof}
Identify the open orbits $(G\times G)[\mathrm{Id}_\Pi] \subset X_\Pi$ and $(G\times G)[\mathrm{Id}_{\Pi'}] \subset X_{\Pi'}$ with the same quotient of $G_\mathrm{ad}$, say $G_1$. Since $\Supp(\lambda) = \Supp(\lambda')$, it follows that $G_1 \cap X^\circ_\Pi$ and $G_1 \cap X^\circ_{\Pi'}$ are both identified with the same open subset $G_1^\circ \subset G_1$. On the other hand, since $X^\circ_\Pi \subset X_\Pi$ and $X^\circ_{\Pi'} \subset X_{\Pi'}$ intersect the respective closed orbits, they intersect every orbit, hence $X_\Pi = (G\times G) X^\circ_\Pi$ and $X_{\Pi'} = (G\times G) X^\circ_{\Pi'}$. Therefore the identity on $G_1$ extends to an (equivariant) morphism $X_\Pi \longrightarrow X_{\Pi'}$ if and only if the identity on $G_1^\circ$ extends to a morphism $X^\circ_\Pi \longrightarrow X^\circ_{\Pi'}$ if and only if $\Bbbk[X^\circ_{\Pi'}] \subset \Bbbk[X^\circ_\Pi]$.

By Lemma \ref{lem I: generatori X-lambda-mu}, the coordinate ring $\Bbbk[X^\circ_\Pi]$ is generated as $\mathfrak{g} \oplus \mathfrak{g}$-algebra by $\Bbbk[X^\circ_\lambda]$ together with the set $\{\phi_\mu/\phi_\lambda\}_{\mu \in \Pi}$, while $\Bbbk[X^\circ_{\Pi'}]$ is generated by $\Bbbk[X^\circ_{\lambda'}]$ together with the set $\{\phi_{\mu'}/\phi_{\lambda'}\}_{\mu' \in \Pi'}$. On the other hand, by Proposition \ref{prop I: morfismi X_lambda} it follows that $\Bbbk[X^\circ_{\lambda'}] \simeq \Bbbk[X^\circ_{\lambda}]$, therefore there exists an equivariant morphism $X_\Pi \to X_{\Pi'}$ if and only if $\phi_{\mu'}/\phi_{\lambda'} \in \Bbbk[X_{\Pi}^\circ]$ for every $\mu' \in \Pi'$. By the description of $\Bbbk[X_{\Pi}^\circ]$ in terms of the projective coordinates of $X_\Pi$, it follows that $\phi_{\mu'}/\phi_{\lambda'} \in \Bbbk[X_{\Pi}^\circ]$ if and only if there exists $n\in \mathbb{N}$ and $\phi \in A_n(\Pi) = \Big(\bigoplus_{\mu \in \Pi} \End(V(\mu))\Big)^n$ such that $\phi_{\mu'} / \phi_{\lambda'} = \phi/\phi_\lambda^n$. On the other hand, $\phi$ has to be $B\times B^-$-semiinvariant, hence the claim follows by Lemma \ref{lem I: coefficientimatriciali}.
\end{proof}

In terms of the semigroup $\Omega(\Pi)$, we may reformulate previous proposition as follows.

\begin{cor} \label{cor I: criterio semigruppi}
Let $\Pi, \Pi'$ be simple subsets with resp. maximal elements $\lambda, \lambda'$ and assume that $\Supp(\lambda) = \Supp(\lambda')$. There exists a $G\times G$-equivariant morphism $X_\Pi \to X_{\Pi'}$ if and only if $\Omega(\Pi') \subset \Omega(\Pi)$ if and only if $\mu' - \lambda' \in \Omega(\Pi)$ for all $\mu' \in \Pi'$.
\end{cor}

\begin{proof}
If $X_\Pi \longrightarrow X_{\Pi'}$, then in particular we have $X_\Pi^\circ \longrightarrow X_{\Pi'}^\circ$: hence by Remark \ref{oss I: Omega(Pi)} it follows that $\Omega(\Pi') \subset \Omega(\Pi)$ and we get $\mu' - \lambda' \in \Omega(\Pi)$ for all $\mu' \in \Pi'$. Suppose conversely that $\mu' - \lambda' \in \Omega(\Pi)$ for all $\mu' \in \Pi'$: then $X_\Pi$ dominates $X_{\Pi'}$ by Proposition \ref{prop I: criterio tensoriale}.
\end{proof}

\begin{dfn}
\begin{itemize}
	\item[i)] Suppose that $\Pi, \Pi' \subset \Lambda^+$ are simple with resp. maximal elements $\lambda, \lambda'$. Then $\Pi$ and $\Pi'$ are \textit{equivalent} (and we write $\Pi \sim \Pi'$) if $\Supp(\lambda') = \Supp(\lambda)$ and $\Pi' - \lambda' = \Pi - \lambda$.
	\item[ii)] A weight $\mu \in \Pi^+(\lambda)$ is \textit{trivial} if $X_{\lambda, \mu} \simeq X_\lambda$ as $G\times G$-varieties. We will denote the subset of the trivial weights in $\Pi^+(\lambda)$ by $\Pi^+_\mathrm{tr}(\lambda)$.
\end{itemize}
\end{dfn}

\begin{oss}	\label{oss I: caratterizzazione-pesi-banali}
\begin{itemize}
	\item[i)] By Proposition \ref{prop I: criterio tensoriale}, a weight $\mu \in \Pi^+(\lambda)$ is trivial if and only if there exists $n\in \mathbb{N}$ such that $V(\mu+(n-1)\lambda) \subset V(\lambda)^{\otimes n}$. Equivalently, $\mu \in \Pi^+(\lambda)$ is trivial if and only if $\Omega(\lambda,\mu) = \Omega(\lambda)$, if and only if $\mu - \lambda \in \Omega(\lambda)$.
	\item[ii)] We may describe the semigroup $\Omega(\lambda)$ in terms of trivial weights as follows:
\[	\Omega(\lambda) = \{\mu - n\lambda \, : \, \mu \in \Pi^+_\mathrm{tr}(n\lambda)\}.	\]
Suppose indeed that $\mu \in \Pi^+_\mathrm{tr}(n\lambda)$: then by Proposition \ref{prop I: morfismi X_lambda} we have $X_{n\lambda} \simeq X_\lambda$, and it follows $\mu - n\lambda \in \Omega(n\lambda) = \Omega(\lambda)$. Conversely, if $\theta \in \Omega(\lambda)$, then $V(n\lambda + \theta) \subset V(\lambda)^{\otimes n}$, hence $n\lambda + \theta \in \Pi^+_\mathrm{tr}(n\lambda)$.
	\item[iii)] Consider the semigroup
\[	\widetilde\Omega(\lambda) = \{\mu - n\lambda \, : \, \mu \in \Pi^+(n\lambda)\}.	\]
Since $X_{n\lambda} \simeq X_\lambda$ for all $n > 0$, considering the normalizations $\widetilde{X}_{n\lambda} \simeq \widetilde{X}_\lambda$ it follows also $X_{\Pi^+(n\lambda)} \simeq X_{\Pi^+(\lambda)}$, hence $\widetilde\Omega(\lambda) = \Omega(\Pi^+(\lambda))$. On the other hand, for every $\mu \in \Pi^+(\lambda)$ there exists $n \in \mathbb{N}$ such that $V(n\mu) \subset V(\lambda)^{\otimes n}$ (see \cite[Lemma 4.9]{AB} or \cite[Lemma 1]{Ti}), therefore $\widetilde\Omega(\lambda) = \Omega(\lambda)_\mathbb{Q} \cap \mathbb{Z}\Delta$ is the \textit{saturation} of $\Omega(\lambda)$ in $\mathbb{Z}\Delta$ (where $\Omega(\lambda)_\mathbb{Q}$ denotes the cone generated by $\Omega(\lambda)$ in $\mathbb{Q}\Delta$). We say that $\Omega(\Pi)$ is \textit{saturated} in $\mathbb Z\Delta$ if $\Omega(\Pi) = \widetilde\Omega(\lambda)$.
	\item[iv)] Suppose that $\Pi$ is simple with maximal element $\lambda$ and let $\pi \in \widetilde\Omega(\lambda)$, then by iii) it exists $n \in \mathbb{N}$ such that $n\pi \in \Omega(\Pi)$. Since $B\times B^-$-semiinvariant functions in $\Bbbk(G)$ are uniquely determined by their weights up to scalar multiples, if $f_\pi, f_{n\pi} \in \Bbbk(G_\mathrm{ad})^{(B\times B^-)}$ are $B\times B^-$-semiinvariant functions of weights $(\pi,\pi^*)$ and $(n\pi,n\pi^*)$, then by Remark \ref{oss I: Omega(Pi)} we have $f_\pi \in \Bbbk[\widetilde{X}^\circ_\lambda]^{(B \times B^-)}$ and $f_{n\pi} \in \Bbbk[X^\circ_\Pi]^{(B \times B^-)}$. Since they have the same weight, $f_{n \pi}$ and $f_\pi^n$ are proportional, hence the normality of $X_\Pi$ implies that $f_\pi \in \Bbbk[X^\circ_\Pi]$, i.e. $\pi \in \Omega(\Pi)$. Together with iii), it follows that the normality of $X_\Pi$ is equivalent to the saturation of $\Omega(\Pi)$ in $\mathbb Z\Delta$.
\end{itemize}
\end{oss}

\begin{cor} \label{cor I: etichette}
Let $\Pi, \Pi' \subset \Lambda^+$ be simple subsets with resp. maximal elements $\lambda$ and $\lambda'$.
\begin{itemize}
	\item[i)] If $\Pi \sim \Pi'$, then $X_\Pi \simeq X_{\Pi'}$ as $G \times G$-varieties. In particular $\Omega(\Pi) = \Omega(\Pi')$.
	\item[ii)] Assume that $\Pi, \Pi'$ both have cardinality 2 and that $X_\Pi \not \simeq X_\lambda$. If $X_{\Pi'} \simeq X_\Pi$ as $G\times G$-varieties, then $\Pi' \sim \Pi$.
\end{itemize}
\end{cor}

\begin{proof}
i) If $\mu\in \Pi$, let $\mu' \in \Pi'$ be such that $\mu-\lambda = \mu' - \lambda'$. Since $B \times B^-$ eigenfunctions in $\Bbbk(G)$ are uniquely determined by their weight up to scalar factors, we have that $\phi_\mu/\phi_\lambda$ and $\phi_{\mu'}/\phi_{\lambda'}$ are proportional. If $\Pi \sim \Pi'$ it follows then by Lemma \ref{lem I: generatori X-lambda-mu} that $\Bbbk[X^\circ_\Pi] \simeq \Bbbk[X^\circ_{\Pi'}]$, and reasoning as in Proposition \ref{prop I: criterio tensoriale} we get an isomorphism of $G\times G$-varieties $X_\Pi \simeq X_{\Pi'}$. The last claim follows by Corollary \ref{cor I: criterio semigruppi}.

ii) Denote $\Pi = \{\lambda, \mu\}$ and $\Pi' = \{\lambda', \mu'\}$ and suppose that $X_{\Pi'} \simeq X_\Pi$. By Proposition \ref{prop I: timashev} ii) it follows that $\Supp(\lambda) = \Supp(\lambda')$, since otherwise $X_\Pi$ and $X_{\Pi'}$ would have non-isomorphic closed orbits. By Proposition \ref{prop I: criterio tensoriale} it follows that $V(\mu'-\lambda' +n\lambda) \subset V(\mu)^{\otimes k} \otimes V(\lambda)^{\otimes n-k}$ for some $n \in \mathbb{N}$ and $k \leqslant n$, so comparing highest weights on the right and on the left we get that $\mu' - \lambda' \leqslant k(\mu-\lambda)$. Since $X_{\Pi'} \not \simeq X_\lambda$, by Proposition \ref{prop I: criterio tensoriale} it must be $k >0$, hence $\mu' - \lambda' \leqslant \mu-\lambda$. An analogous argument shows that $\mu - \lambda \leqslant \mu' - \lambda'$, and the claim follows.
\end{proof}

\begin{oss}	\label{oss I: etichette}
Together with Remark \ref{oss I: Omega(Pi)}, previous corollary shows that the set
\[	\widetilde\Omega(\lambda) \smallsetminus \Omega(\lambda) = \{ \mu - n\lambda \, : \, \mu \in \Pi^+(n\lambda) \smallsetminus \Pi_{\mathrm{tr}}^+(n\lambda) \}	\]
classifies the simple linear compactifications $X_\Pi$ such that $\widetilde X_\lambda \to X_\Pi \to X_\lambda$ and $\mathrm{card}(\Pi) = 2$ up to equivariant isomorphism. If indeed $X_\Pi$ is such a compactification and if $X_\Pi \not \simeq X_\lambda$, then it must be $\Pi = \{\lambda', \mu'\}$ for some $\lambda' \in \Lambda^+$ with $\Supp(\lambda') = \Supp(\lambda)$ and some $\mu' \in \Pi^+(\lambda') \smallsetminus \Pi^+_{\mathrm{tr}}(\lambda')$, and by Corollary \ref{cor I: etichette} $X_\Pi$ is uniquely determined by the difference $\mu' - \lambda'$. On the other hand up to consider an equivalent simple subset we may assume that  $\lambda' = n\lambda$ for some $n \in \mathbb{N}$, therefore $\mu' - \lambda' \in \widetilde{\Omega}(\lambda) \smallsetminus \Omega(\lambda)$.
\end{oss}

Suppose that $\Pi\subset \Lambda^+$ is simple with maximal element $\lambda$. Then, by the isomorphism $\widetilde{X}_\lambda \simeq X_{\Pi^+(\lambda)}$, Proposition \ref{prop I: criterio tensoriale} yields as well a tensorial criterion of normality for $X_\Pi$: $X_\Pi$ is normal if and only if, for every $\nu \in \Pi^+(\lambda)$, they exist $\mu_1, \ldots, \mu_n \in \Pi$ such that $V(\nu+(n-1)\lambda) \subset V(\mu_1) \otimes \ldots \otimes V(\mu_n)$. As shown in \cite{BGMR}, this characterization turns out to be equivalent to a combinatorial property of $\Pi$.

\begin{dfn}[{\cite[Def. 2.7]{BGMR}}]	\label{def: twin}
If $\Delta' \subset \Delta$ is a non-simply laced connected component, order the simple roots in $\Delta'= \{ \alpha_1, \ldots, \alpha_r\}$ starting from the extreme of the Dynkin diagram of $\Delta'$ which contains a long root and denote $\alpha_q$ the first short root in $\Delta'$. If $\lambda \in \Lambda^+$ is such that $\alpha_q\not \in \Supp(\lambda)$ and such that $\Supp(\lambda)\cap \Delta'$ contains a long root, denote $\alpha_p$ the last long root which occurs in $\Supp(\lambda)\cap \Delta'$; for instance, if $\Delta'$ is not of type $\mathsf{G}_2$, then the numbering is as follows:
\[	\begin{picture}(9000,1800)(2000,-900)
           \put(0,0){\multiput(0,0)(3600,0){2}{\circle*{150}}\thicklines\multiput(0,0)(2500,0){2}{\line(1,0){1100}}\multiput(1300,0)(400,0){3}{\line(1,0){200}}}
           \put(3600,0){\multiput(0,0)(3600,0){2}{\circle*{150}}\thicklines\multiput(0,0)(2500,0){2}{\line(1,0){1100}}\multiput(1300,0)(400,0){3}{\line(1,0){200}}}
           \put(7200,0){\multiput(0,0)(1800,0){2}{\circle*{150}}\thicklines\multiput(0,-60)(0,150){2}{\line(1,0){1800}}\multiput(1050,0)(-25,25){10}{\circle*{50}}\multiput(1050,0)(-25,-25){10}{\circle*{50}}}
           \put(9000,0){\multiput(0,0)(3600,0){2}{\circle*{150}}\thicklines\multiput(0,0)(2500,0){2}{\line(1,0){1100}}\multiput(1300,0)(400,0){3}{\line(1,0){200}}}
           \put(-150,-700){\tiny $\alpha_1$}
           \put(3450,-700){\tiny $\alpha_p$}
           \put(8850,-700){\tiny $\alpha_q$}
           \put(12450,-700){\tiny $\alpha_r$}
\end{picture}	\]
The \textit{little brother} of $\lambda$ with respect to $\Delta'$ is the dominant weight
\[
\lambda_{\Delta'}^\mathrm{lb} = \lambda - \sum_{i=p}^q \alpha_i =
\left\{ \begin{array}{ll}
		\lambda-\omega_1+\omega_2 & \textrm{ if $G$ is of type $\sf{G}_2$} \\
		\lambda + \omega_{p-1} - \omega_{p}  + \omega_{q+1} & \textrm{ otherwise}
\end{array} \right.
\]
where $\omega_i$ is the fundamental weight associated to $\alpha_i$ if $1\leqslant i \leqslant r$, while we set $\omega_0 = \omega_{r+1} = 0$. We will denote the set of the little brothers of $\lambda$ by $\mathrm{LB}(\lambda)$. Notice that $\mathrm{LB}(\lambda)$ is empty if and only if $\lambda$ satisfies the following condition:
\[	\begin{array}{ll}
	(\star)  &	\begin{array}{l}	\text{For every non-simply laced connected component $\Delta'\subset \Delta$, if $\Supp(\lambda) \cap \Delta'$ contains a}\\	\text{long root, then it contains also the short root which is adjacent to a long simple root.}	\end{array}	\end{array}	\]
In case $\Delta$ is connected and $\lambda$ does not satisfy $(\star)$, then we set $\lambda^\mathrm{lb} = \lambda_\Delta^\mathrm{lb}$.
\end{dfn}

\begin{teo} [{\cite[Thm.~2.10]{BGMR}}] 	\label{teo I: BGMR}
Let $G$ be a semisimple group and let $\Pi\subset \Lambda^+$ be simple with maximal element $\lambda$. Then the variety $X_\Pi$ is normal if and only if $\mathrm{LB}(\lambda)\subset \Pi$. In particular, $X_\lambda$ is normal if and only if $\lambda$ satisfies $(\star)$.
\end{teo}

\begin{cor} \label{cor I: pesi banali caso star}
Suppose that $\lambda \in \Lambda^+$ satisfies $(\star)$. Then $X_\Pi \simeq X_\lambda$ for every simple subset $\Pi$ with maximal element $\lambda$. In particular $\Pi^+_\mathrm{tr}(\lambda) = \Pi^+(\lambda)$.
\end{cor}

\begin{proof}
Let $\Pi \subset \Lambda^+$ be simple with maximal element $\lambda$. Then the normalization of $X_\lambda$ factors through $X_\Pi$, so the claim follows by the normality of $X_\lambda$.
\end{proof}

\subsection{Some remarks on tensor product decompositions}

We conclude this section with some explicit results on tensor products that will be needed in the following.

\begin{lem}	\label{lem I: riduzione-levi}
Let $\lambda_1, \ldots, \lambda_n \in \Lambda^+$ and let $L$ be a Levi subgroup of $G$. Denote $\Lambda^+_L \subset \Lambda$ the semigroup of dominant weights for $L$ and, for $\pi \in \Lambda^+_L$, denote by $V_L(\pi)$ the simple $L$-module of highest weight $\pi$. If $\mu \in \Lambda^+_L$ is such that $V_L(\mu) \subset V_L(\lambda_1) \otimes \ldots \otimes V_L(\lambda_n)$, then $\mu \in \Lambda^+$ and $V(\mu) \subset V(\lambda_1) \otimes \ldots \otimes V(\lambda_n)$.
\end{lem}

\begin{proof}
Denote $\Delta' \subset \Delta$ the set of simple roots associated to $L$. Since $\mu \in \Lambda^+_L$, we have $\langle \mu, \alpha^\vee \rangle \geqslant 0$ for every $\alpha \in \Delta'$. On the other hand, being $V_L(\mu) \subset V_L(\lambda_1) \otimes \ldots \otimes V_L(\lambda_n)$, we have $\sum_i \lambda_i - \mu \in \mathbb N \Delta'$: hence we get $\langle \mu, \alpha^\vee \rangle \geqslant \langle \sum_i \lambda_i, \alpha^\vee \rangle$ for every $\alpha \in \Delta \smallsetminus \Delta'$, and it follows that $\mu \in \Lambda^+$ since $\sum_i \lambda_i$ is so.

We now prove the second claim by induction on $n$, the basis being the case $n=2$ (see \cite[Lemma~2.4]{BGMR}). If $\mathfrak a$ is any Lie algebra, denote $\mathfrak U(\mathfrak a)$ the corresponding universal enveloping algebra. Assume $n=2$ and regard $V_L(\lambda_1) \otimes V_L(\lambda_2) \subset V(\lambda_1) \otimes V(\lambda_2)$. Fix maximal vectors $v_1 \in V(\lambda_1)$ and $v_2 \in V(\lambda_2)$ for $B$ and let $p \in \mathfrak U(\mathfrak l\cap\mathfrak u^-) \otimes \mathfrak U(\mathfrak l\cap\mathfrak u^-)$ be such that $p\,(v_1 \otimes v_2) \in V_L(\lambda_1) \otimes V_L(\lambda_2)$ is a maximal vector of weight $\mu$ for the Borel $B \cap L \subset L$: to prove the claim we only need to show that $p\,(v_1 \otimes v_2)$ is a maximal vector for $B$ too. If $\alpha \in \Delta'$ then we have $e_\alpha p\,(v_1 \otimes v_2) = 0$ by hypothesis. On the other hand, if $\alpha \in \Delta \smallsetminus \Delta'$, then $e_\alpha$ commutes with $p$, since by its definition $p$ is supported only on the $f_\alpha$'s with $\alpha \in \Delta'$. Since $v_1 \otimes v_2$ is a maximal vector for $B$, we get then
$e_\alpha p\,(v_1 \otimes v_2) = p\, e_\alpha (v_1 \otimes v_2) = 0$, therefore $p\,(v_1 \otimes v_2)$ generates a simple $G$-module of highest weight $\mu$.

Suppose now $n>2$ and let $\mu' \in \Lambda^+_L$ be such that $V_L(\mu') \subset V_L(\lambda_1) \otimes \ldots \otimes V_L(\lambda_{n-1})$ and $V_L(\mu) \subset V_L(\mu') \otimes V_L(\lambda_n)$. Then $\mu' \in \Lambda^+$ by the first part of the proof, while by the inductive hypothesis we get $V(\mu') \subset V(\lambda_1) \otimes \ldots \otimes V(\lambda_{n-1})$ and $V(\mu) \subset V(\mu') \otimes V(\lambda_n)$, so the claim follows.
\end{proof}

\begin{cor}	\label{cor I: differenza tipo A}
Let $\lambda \in \Lambda^+$ and let $\mu \in \Pi^+(\lambda)$ be such that $\Supp_\Delta(\lambda-\mu)$ is simply laced regarded as a subset of the vertices of the Dynkin diagram of $G$. Then $\mu \in \Pi^+_\mathrm{tr}(\lambda)$.
\end{cor}

\begin{proof}
Denote $L$ the Levi subgroup associated to $\Delta' = \Supp_\Delta(\lambda-\mu)$. By Corollary \ref{cor I: pesi banali caso star} applied to the semisimple part of $L$ it follows that $\mu \in \Pi^+(\lambda)$ is trivial w.r.t. $L$, hence by Remark \ref{oss I: caratterizzazione-pesi-banali} it follows that there exists $n\in \mathbb{N}$ such that $V_L(\mu + (n-1)\lambda) \subset V_L(\lambda)^{\otimes n}$. Therefore by Lemma \ref{lem I: riduzione-levi} we get $V(\mu + (n-1)\lambda) \subset V(\lambda)^{\otimes n}$, and by Remark \ref{oss I: caratterizzazione-pesi-banali} i) it follows that $\mu \in \Pi^+_\mathrm{tr}(\lambda)$.
\end{proof}

Let $n \in \mathbb{N}$ and consider the set
\[	\mathrm{Tens}_n(G) = \{(\lambda_0, \ldots, \lambda_n) \in (\Lambda^+)^{n+1} \, : \, V(\lambda_0) \subset V(\lambda_1) \otimes \ldots \otimes V(\lambda_n)\}.	\]
Following lemma has been proved in several references, usually in the case $n=2$. Since we will need that, we claim it in a slightly more general form, which is easily reduced to the case $n=2$ proceeding by induction on $n$.

\begin{lem}	[{\cite[Lemma 3.9]{Ku}}]	\label{lem I: tensor semigroup}
The set $\mathrm{Tens}_n(G)$ is a semigroup with respect to the addition.
\end{lem}

An easy application of previous lemma which will be very useful for us is the following.

\begin{cor}	\label{cor I: traslazione}
Let $\lambda_0, \ldots, \lambda_n, \in \Lambda^+$ be such that $V(\lambda_0) \subset V(\lambda_1) \otimes \ldots \otimes V(\lambda_n)$. Then, for any $\mu \in \Lambda^+$, it also holds $V(\lambda_0 + \mu) \subset V(\lambda_1 + \mu) \otimes V(\lambda_2) \otimes \ldots \otimes V(\lambda_n)$.
\end{cor}

\begin{cor}	\label{cor I: differenza radice lunga}
Let $\lambda, \mu \in \Lambda^+$ and let $\nu \in \Pi^+(\mu)$ be such that $\Supp_\Delta(\mu-\nu) \cap \Supp(\lambda) \neq \varnothing$, suppose moreover that $\Supp_\Delta(\mu-\nu)$ is connected and that $\mu - \nu$ is the highest root of the root subsystem generated by $\Supp_\Delta(\mu-\nu)$. Then $V(\lambda + \nu) \subset V(\mu) \otimes V(\lambda)$.
\end{cor}

\begin{proof}
Denote $L$ the Levi subgroup associated to $\Supp_\Delta(\mu-\nu)$ and denote $\mathfrak{l}$ its Lie algebra. Consider $\mu - \nu$: by the assumption on $\mu - \nu$, we have an isomorphism of $\mathfrak{l}$-modules $V_L(\mu - \nu) \simeq \mathfrak{l}$. Since $\Supp_\Delta(\mu-\nu) \cap \Supp(\lambda) \neq \varnothing$, the $\mathfrak{l}$-action induces a non-zero morphism $V_L(\mu-\nu) \otimes V_L(\lambda) \rightarrow V_L(\lambda)$ which is surjective by irreducibility, hence we get $V_L(\lambda) \subset V_L(\mu-\nu) \otimes V_L(\lambda)$. By Corollary \ref{cor I: traslazione} this implies $V_L(\lambda + \nu) \subset V_L(\mu) \otimes V_L(\lambda)$, and the claim follows applying Lemma \ref{lem I: riduzione-levi}.
\end{proof}

We now describe an explicit result which we will need in the special case $G = \mathrm{Spin}(2r+1)$, which we will treat in the rest of the paper. Set $\Delta = \{\alpha_1, \ldots, \alpha_r\}$ and denote $\omega_1, \ldots, \omega_r$ the fundamental weights. For convenience, we also denote $\varpi_k = \sum_{j=1}^{k-1} j\alpha_j + k\sum_{j=k}^r \alpha_j$: therefore if $0 < k < r$ we have $\varpi_k = \omega_k$, whereas $\varpi_0 = 0$ and $\varpi_r = 2\omega_r$.

\begin{lem}	\label{lem I: differenza radice positiva}
Let $G = \mathrm{Spin}(2r+1)$. Let $\lambda, \mu \in \Lambda^+$ and $\nu \in \Pi^+(\mu)$.
\begin{itemize}
	\item[i)] If $\mu - \nu \in \Phi^+_l(\lambda)$, then $V(\lambda + \nu) \subset V(\mu) \otimes V(\lambda)$.
	\item[ii)] If $\alpha_r \in \Supp(\lambda+\nu)$ and if $\mu - \nu \in \Phi^+(\lambda)$, then $V(\lambda + \nu) \subset V(\mu) \otimes V(\lambda)$.
	\item[iii)] If $\Supp(\lambda) \neq \{\alpha_r\}$ and if $\mu - \nu \in 2\Phi_s^+(\lambda)$, then $V(\lambda + \nu) \subset V(\mu) \otimes V(\lambda)$.
\end{itemize}
\end{lem}

\begin{proof}
Denote $\theta = \mu - \nu$ and set $\Supp_\Delta(\theta) = \{ \alpha_{p+1}, \ldots, \alpha_q \}$, where $0 \leqslant p < q \leqslant r$.

i) Notice that it holds one of the followings:
\begin{itemize}
	\item[-] $q < r$ and $\theta = \sum_{i=p+1}^{q}\alpha_i$;	
	\item[-] $q=r$ and $\theta = \sum_{i=p+1}^{k}\alpha_i + 2 \sum_{i=k+1}^r \alpha_i$ for some $k$ with $p < k < r$ .
\end{itemize}
Suppose that we are in the first case: then $\theta$ is the highest root of the subsystem generated by $\Supp_\Delta(\theta)$ and the claim follows by Corollary \ref{cor I: differenza radice lunga}. Suppose that we are in the second case: then we have $\mu = \nu - \varpi_p + \varpi_{p+1} - \varpi_k + \varpi_{k+1}$. Since $\mu$ is dominant, it must be $\alpha_p, \alpha_k \in \Supp(\nu)$. Notice also that by Lemma \ref{lem I: riduzione-levi} we may assume that $\Supp_\Delta(\theta) = \Delta$, i.e. $p=0$. Let $\alpha_j \in \Supp(\lambda)$: then by applying Corollary \ref{cor I: traslazione} twice (first with $\nu-\omega_k$ and then with $\lambda-\omega_j$) we are reduced to the following inclusion, which can be checked directly:
\begin{itemize}
	\item[-] If $1\leqslant j \leqslant r$ and $1\leqslant k \leqslant r-1$, then $V(\varpi_k + \varpi_j) \subset V(\varpi_j)\otimes V(\varpi_1 + \varpi_{k+1})$.
\end{itemize}

ii) By part i), we only need to consider the case where $\theta$ is a positive short root. Notice that $\theta = \sum_{i=p+1}^r \alpha_i$ for some $p$ with $0 \leqslant p <r$. By Lemma \ref{lem I: riduzione-levi} we may assume $p=0$, so we have that $\mu = \nu + \omega_1$. Suppose that $\alpha_r \in \Supp(\lambda)$: then by applying Corollary \ref{cor I: traslazione} twice (first with $\nu$ and then with $\lambda-\omega_r$) we are reduced to the following inclusion, which can be checked directly:
\begin{itemize}
	\item[-] $V(\omega_r) \subset V(\omega_1)\otimes V(\omega_r)$;
\end{itemize}
Suppose now that $\alpha_r \in \Supp(\nu)$ and let $\alpha_j \in \Supp(\lambda)$: then by applying Corollary \ref{cor I: traslazione} twice (first with $\nu-\omega_r$ and then with $\lambda-\omega_j$) we are reduced to the following inclusion, which can be checked directly:
\begin{itemize}
	\item[-] If $1\leqslant j \leqslant r$, then $V(\omega_j + \omega_r) \subset V(\omega_j)\otimes V(\omega_1 + \omega_r)$. 
\end{itemize}	

iii) Notice that $\theta = 2\sum_{i=p+1}^r \alpha_i$ for some $p$ with $0 \leqslant p <r$. By Lemma \ref{lem I: riduzione-levi} we may assume $p=0$, so we have that $\mu = \nu + 2\omega_1$. Let $\alpha_j \in \Supp(\lambda)$ with $j < r$: then by applying Corollary \ref{cor I: traslazione} twice (first with $\nu$ and then with $\lambda-\omega_j$) we are reduced to the following inclusion, which can be checked directly:
\begin{itemize}
	\item[-] If $1\leqslant j <r$, then $V(\omega_j) \subset V(2\omega_1) \otimes V(\omega_j)$. 
\qedhere
\end{itemize}
\end{proof}

\section{Trivial weights in the odd orthogonal case}

From now on we will suppose $G = \mathrm{Spin}(2r+1)$. Set $\Delta = \{\alpha_1, \ldots, \alpha_r\}$ and denote $\omega_1, \ldots, \omega_r$ the fundamental weights. For convenience, if $0 \leqslant k \leqslant r$ we also denote $\varpi_k = \sum_{j=1}^{k-1} j\alpha_j + k\sum_{j=k}^r \alpha_j$: therefore $\varpi_k = \omega_k$ if $0 < k < r$, whereas $\varpi_0 = 0$ and $\varpi_r = 2\omega_r$. If $\lambda \in \Lambda^+$ is non-zero we denote by $q(\lambda)$ the maximum such that $\langle \lambda, \alpha_{q(\lambda)}^\vee \rangle \neq 0$, while we set $q(\lambda) = 0$ if $\lambda = 0$. We are going to prove the following combinatorial characterization of trivial weights, the rest of the section will be devoted to its proof.

\begin{teo}	\label{teo II: classificazione pesi banali B}
Suppose that $G = \mathrm{Spin}(2r+1)$. Let $\lambda \in \Lambda^+$, $\mu \in \Pi^+(\lambda)$ and denote $\lambda-\mu = \sum_{i=1}^r a_i \alpha_i$. Then $\mu \in \Pi^+_\mathrm{tr}(\lambda)$ if and only if $a_r$ is even or $a_r > 2\min\{r-q(\lambda), r-q(\mu)\}$.
\end{teo}

\begin{oss}	\label{oss II: l(theta)}
Let $\lambda \in \Lambda^+$ and $\mu \in \Pi^+(\lambda)$. Set $\theta = \lambda - \mu \in \mathbb N \Delta$, say $\theta = \sum_{i=1}^r a_i \alpha_i$, and denote $l(\theta)\leqslant r$ the minimum such that $a_i = a_r$ for every $i \geqslant l(\theta)$. Since $\mu$ is dominant, it follows that $\Supp(\theta^+) \subset \Supp(\lambda)$. If $q(\lambda) < i < r$, then we have $a_{i-1} - 2a_i + a_{i+1} = \langle \mu, \alpha_i^\vee \rangle \geqslant 0$, whereas if $q(\lambda) < r$ then $2a_{r-1} - 2a_r = \langle \mu, \alpha_r^\vee \rangle \geqslant 0$. In particular this implies $a_{q(\lambda)} \geqslant a_{q(\lambda)+1} \geqslant \ldots \geqslant a_r$ and it follows that $\max\{q(\lambda), q(\mu)\} = \max\{l(\theta), q(\lambda)\}$.
\end{oss}

By Corollary \ref{cor I: etichette}, the triviality of $\mu$ depends only on the equivalence class of the simple subset $\{\lambda,\mu\}$. Therefore we may restate previous theorem as follows, not dealing with $\lambda$ but just with the semigroup $\Omega(\lambda)$, which depends only on $\Supp(\lambda)$. 

\begin{cor} \label{cor II: descrizione anello delle coordinate B}
Suppose that $G = \mathrm{Spin}(2r+1)$ and let $\lambda\in \Lambda^+$. Then
\[
	- \Omega(\lambda) = \left\{ \theta \doteq \sum_{i=1}^r a_i \alpha_i  \in \mathbb{N} \Delta \, : \, \begin{array}{c}
	\Supp(\theta^+) \subset \Supp(\lambda) \text{ and } \\
	a_r \text{ is even or } a_r > 2\min\{r- l(\theta), r- q(\lambda)\}	\end{array}	\right\}	
\]
where $l(\theta)\leqslant r$ denotes the minimum such that $a_i = a_r$ for every $i \geqslant l(\theta)$.
\end{cor}

\begin{proof}
By Remark \ref{oss I: caratterizzazione-pesi-banali} ii) we have $-\Omega(\lambda) = \{n\lambda - \mu \, : \, \mu \in \Pi^+_\mathrm{tr}(n\lambda) \}$. Let $\theta = \sum_{i=1}^r a_i \alpha_i \in \mathbb{N}\Delta$. If $\theta \in - \Omega(\lambda)$, then by previous theorem together with Remark \ref{oss II: l(theta)} we get that $a_r$ is even or that $a_r > 2\min\{r-l(\theta), r- q(\lambda)\}$. Conversely if $\Supp(\theta^+) \subset \Supp(\lambda)$ then $n\lambda -\theta$ is dominant for some $n \in \mathbb{N}$, and if moreover $a_r$ is even or $a_r > 2\min\{r-l(\theta), r-q(\lambda)\} = 2\min\{r-q(n\lambda), r-q(n\lambda - \theta)\}$, then by previous theorem we have $n\lambda - \theta \in \Pi^+_\mathrm{tr}(n\lambda)$, hence $\theta \in -\Omega(\lambda)$.
\end{proof}

\begin{oss} \label{oss II: omega_r-1}
Suppose that $\Supp(\lambda) = \{\alpha_{r-1}\}$. Then previous corollary implies that $\mathrm{SO}(2r+1)$ admits a unique non-normal linear compactification $X$ such that $\widetilde{X}_\lambda \to X \to X_\lambda$, namely $X_\lambda$. If indeed $\mu = \lambda - \sum_{i=1}^r a_i \alpha_i \in \Pi^+(\lambda) \smallsetminus \Pi^+_\mathrm{tr}(\lambda)$, then by Corollary \ref{cor II: descrizione anello delle coordinate B} it must be $a_{r-1} = a_r = 1$, and it follows $a_1 = \ldots = a_{r-2} = 0$.
\end{oss}

We now prove Theorem \ref{teo II: classificazione pesi banali B}, the proof will be split in several lemmas. If $\alpha_r \in \Supp(\lambda)$, then by Corollary \ref{cor I: pesi banali caso star} we have $\Pi^+_\mathrm{tr}(\lambda) = \Pi^+(\lambda)$. Therefore throughout this section we will assume that $\alpha_r \not \in \Supp(\lambda)$. First we will prove that the conditions are necessary (Corollary \ref{cor II: necessità}). A basic case is that of the first fundamental weight, treated in the following proposition, where we deduce the isotypic decomposition of the tensor powers of the the standard representation by the Schur-Weyl duality (see for instance \cite[Appendix F]{GW}).

\begin{prop} \label{prop II: schur-weyl1}
Suppose that $G = \mathrm{Spin}(2r+1)$ and let $n \in \mathbb{N}$. If $\mu \in \Pi^+(n\omega_1)$, denote $n\omega_1 -\mu = \sum_{i=1}^r a_i \alpha_i$. Then $V(\mu) \subset V(\omega_1)^{\otimes n}$ if and only if $a_r$ is even or $a_r > 2(r-q(\mu))$.
\end{prop}

\begin{proof}
Embed standardly $\mathrm{SO}(2r+1)$ in $\mathrm{GL}(2r+1)$ and denote $\mathfrak{h} \subset \widetilde{\mathfrak{h}}$ the respective Cartan subalgebras of diagonal matrices. Denote $\varepsilon_1, \ldots, \varepsilon_{2r+1}$ the basis of $\widetilde{\mathfrak{h}}^*$ defined by $\varepsilon_i(\diag(a_1,\ldots,a_{2r+1})) = a_i$ and, for any weight $\lambda = \sum_{i=1}^r \lambda_i \varepsilon_i$, denote $|\lambda| = \sum_{i=1}^r \lambda_i$. With respect to this basis $\mu$ is expressed as follows
$$
	\mu = (n-a_1)\varepsilon_1 + \sum_{i=2}^r (a_{i-1} - a_i) \varepsilon_i.
$$
By the Schur-Weyl duality it follows that $ V(\mu) \subset V(\omega_1)^{\otimes n}$ if and only if $\mu$ extends to a dominant weight $\tilde{\mu} = \sum_{i=1}^{2r+1} \tilde{\mu}_i \varepsilon_i \in \widetilde{\mathfrak{h}}^*$ such that
\[	\left\{ \begin{array}{l}
		|\tilde{\mu}|\leqslant n \\
		|\tilde{\mu}| \equiv n \quad \mathrm{mod} \;2 \\
		\tilde{\mu}_1^t + \tilde{\mu}_2^t \leqslant 2r+1
	\end{array} \right.	\]
where $\tilde{\mu}^t = (\tilde{\mu}_1^t, \ldots, \tilde{\mu}_{\tilde{\mu}_1}^t)$ denotes the transposed of $\tilde{\mu} = (\tilde{\mu}_1, \ldots, \tilde{\mu}_{2r+1})$ regarded as a partition of $|\tilde{\mu}|$.

Suppose that $V(\mu) \subset V(\omega_1)^{\otimes n}$ and let $\tilde{\mu} \in \widetilde{\mathfrak h}^*$ be an extension of $\mu$ as above. Then either $\tilde{\mu}_i = 0$ for every $i>r$ or
\[	\tilde{\mu} = \sum_{i=1}^{\tilde{\mu}_2^t} \tilde{\mu}_i \varepsilon_i + \sum_{i=\tilde{\mu}_2^t+1}^{\tilde{\mu}_1^t} \varepsilon_i	\]
with $\tilde{\mu}_i \geqslant 2$ for $i \leqslant \tilde{\mu}_2^t$. Suppose that $\tilde{\mu}$ is of the first type: then $a_r$ is even since $ a_r = n - |\mu| = n - |\tilde{\mu}| \equiv 0$ mod 2. If instead $\tilde{\mu}$ is of the second type, then $\tilde{\mu}_1^t + \tilde{\mu}_2^t \leqslant 2r+1$ implies $q(\mu) = 2r + 1 - \tilde{\mu}_1^t$ and we get $a_r > 2(r-q(\mu))$ since 
\[	a_r = n - |\mu| = n - |\tilde{\mu}| + 2(\tilde{\mu}_1^t - r - 1) + 1 = n - |\tilde{\mu}| + 2(r -q(\mu)) + 1.	\]

Suppose conversely that $a_r$ is even or that $a_r > 2(r-q(\mu))$, let's show that $V(\mu) \subset V(\omega_1)^{\otimes n}$. Define the weight $\tilde{\mu} \in \widetilde{\mathfrak{h}}^*$ as follows:
\[	\tilde{\mu} = \left\{ \begin{array}{ll}
		\sum_{i=1}^{q(\mu)} \mu_i \varepsilon_i & \mathrm{if \;} a_r \mathrm{\; is \; even}\\
		\sum_{i=1}^{q(\mu)} \mu_i \varepsilon_i + \sum_{i=q(\mu)+1}^{2r-q(\mu)+1} \varepsilon_i & \mathrm{if \;} a_r > 2(r-q(\mu)) \mathrm{\; is \; odd}
	       \end{array} \right.	\]
Then $\tilde{\mu}$ satisfies the conditions of the Schur-Weyl duality and it follows $V(\mu) \subset V(\omega_1)^{\otimes n}$.
\end{proof}

Following lemma will allow us to deduce the necessity of the conditions in Theorem \ref{teo II: classificazione pesi banali B} from the case $\Supp(\lambda) = \{\alpha_1\}$.

\begin{lem}	\label{lem II: semigruppi mu<lambda}
Let $\lambda \in \Lambda^+$ and $\mu \in \Pi^+_\mathrm{tr}(\lambda)$. Then $\Bbbk[X^\circ_\mu] \subset  \Bbbk[X^\circ_\lambda]_{(\phi_\mu/\phi_\lambda)}$. In particular $\Omega(\mu)\subset \Omega(\lambda)_{\lambda-\mu}$, where the latter denotes the semigroup generated in $\mathbb{Z} \Delta$ by $\Omega(\lambda)$ together with $\lambda-\mu$.
\end{lem}

\begin{proof}
Since $X_\lambda \simeq X_{\lambda, \mu}$, it follows that $X_\lambda$ is endowed with a linearized ample line bundle $\mathcal{L}$ possessing a $B\times B^-$-semiinvariant section $s_\mu$ of weight $(\mu,\mu^*)$ which generates a submodule of $\Gamma(X_\lambda,\mathcal{L})$ isomorphic to $\End(V(\mu))$. Correspondingly we get a rational application $X_\lambda \dashrightarrow X_\mu$ which is regular in the affine set $\big(X^\circ_\lambda\big)_{(\phi_\mu/\phi_\lambda)} \subset X_\lambda^\circ$ defined by the non-vanishing of $\phi_\mu/\phi_\lambda \in \Bbbk[X^\circ_\lambda]$, and it follows that $\Bbbk[X^\circ_\mu] \subset  \Bbbk[X^\circ_\lambda]_{(\phi_\mu/\phi_\lambda)}$. The second claim follows by the first one applying Remark \ref{oss I: Omega(Pi)}.
\end{proof}

\begin{cor}	\label{cor II: necessità}
Suppose that $G= \mathrm{Spin}(2r+1)$ and let $\lambda \in \Lambda^+$ be such that $\alpha_r \not \in \Supp(\lambda)$. Let $\mu \in \Pi^+_\mathrm{tr}(\lambda)$ and denote $\lambda - \mu = \sum_{i=1}^r a_i \alpha_i$, then either $a_r$ is even or $a_r > 2\min\{r-q(\lambda), r-q(\mu)\}$.
\end{cor}

\begin{proof}
Since $\alpha_r \not \in \Supp(\lambda)$, we have $\lambda \in \mathbb{Z}\Delta$ and there exists $n>0$ such that $\lambda \leqslant n\omega_1$ with $\Supp_\Delta(n\omega_1 - \lambda) \subset \{\alpha_1, \ldots, \alpha_{q(\lambda)-1}\}$. Since $\alpha_r \not \in \Supp_\Delta(n\omega_1 - \lambda)$, we have $V(\lambda) \subset V(\omega_1)^{\otimes n}$, hence $\lambda - n\omega_1 \in \Omega(\omega_1)$ by Remark \ref{oss I: caratterizzazione-pesi-banali} ii). On the other hand by Corollary \ref{cor I: criterio semigruppi} we have $\Omega(\omega_1) = \Omega(n\omega_1)$, hence $\lambda \in \Pi^+_\mathrm{tr}(n\omega_1)$ by Remark \ref{oss I: caratterizzazione-pesi-banali} i). By Lemma \ref{lem II: semigruppi mu<lambda} we get then $\mu - \lambda \in \Omega(\lambda) \subset \Omega(\omega_1)_{n\omega_1 - \lambda}$, hence by Remark \ref{oss I: caratterizzazione-pesi-banali} ii) they exist $k,m \in \mathbb{N}$ and $\mu' \in \Pi^+_\mathrm{tr}(m\omega_1)$ such that $\mu - \lambda = \mu' - m\omega_1 + k(n\omega_1 - \lambda)$.

Denote $m\omega_1 - \mu' = \sum a_i' \alpha_i$. By the definition of $n$ it follows that $a_i = a'_i$ for all $i \geqslant q(\lambda)$: hence by Proposition \ref{prop II: schur-weyl1} either $a_r$ is even or $a_r > 2 \min\{r-1, r-q(\mu')\}$. If $\mu' = 0$, then $q(\mu') = 0$, hence $a_r > 2(r-1) \geqslant 2(r-q(\lambda))$. Suppose instead $\mu' \neq 0$ and assume that $q(\mu') > q(\lambda)$: then it must be $q(\mu') = q(\mu)$, therefore we have $2(r-q(\mu')) \geqslant 2\min\{r-q(\lambda), r-q(\mu)\}$.
\end{proof}

We now show that the conditions of Theorem \ref{teo II: classificazione pesi banali B} are sufficient. We distinguish three different cases:
\begin{itemize}
	\item[i)]	$a_{r-1} \neq a_r$, i.e. $\alpha_r \in \Supp(\mu)$ (Lemma \ref{lem II: mu-soddisfa-star}).
	\item[ii)]	$a_{r-1} = a_r$ is even (Lemma \ref{lem II: ar-pari}).
	\item[iii)] $a_{r-1} = a_r > 2 \min\{r-q(\lambda), r-q(\mu)\}$ is odd (Lemma \ref{lem II: ar-dispari}).
\end{itemize}

\begin{lem} \label{lem II: mu-soddisfa-star}
Let $\lambda \in \Lambda^+$ be such that $\alpha_r \not \in \Supp(\lambda)$ and let $\mu \in \Pi^+(\lambda)$. If $\alpha_r \in \Supp(\mu)$, then $\mu \in \Pi^+_\mathrm{tr}(\lambda)$.  
\end{lem}

\begin{proof}
We proceed by induction on $a_{r-1}a_r$. Suppose that either $a_{r-1}= 0$ or $a_r = 0$: then $\Supp_\Delta(\lambda - \mu)$ has all components of type ${\sf A}$ and the claim follows by Corollary \ref{cor I: differenza tipo A}. Suppose now that $a_{r-1}$ and $a_r$ are both non-zero and denote $p < r-1$ the maximum such that $a_p = 0$, or set $p=0$ if $a_i\neq 0$ for all $i$. Define
\[	\mu' = \mu + \sum_{i = p+1}^r \alpha_i = \mu - \varpi_p + \varpi_{p+1}.	\]
Notice that $\mu'$ is dominant: if indeed $p>0$, then $\alpha_p \in \Supp(\mu)$ since $\langle \mu, \alpha_p^\vee \rangle \geqslant \langle \lambda, \alpha_p^\vee \rangle + a_{p+1} > 0$. Therefore $\mu' \in \Pi^+(\lambda)$ and by construction we have $\mu < \mu'$. Hence by Lemma \ref{lem I: differenza radice positiva} ii) it follows $V(\lambda + \mu)\subset V(\lambda) \otimes V(\mu')$ and we get $\mu-\lambda \in \Omega(\lambda,\mu')$.

Consider now $\mu'$ and denote $\lambda - \mu' = \sum a'_i \alpha_i$: then $\alpha_r \in \Supp(\mu')$ and $a'_{r-1}a'_r < a_{r-1}a_r$, so by the inductive hypothesis it follows that $\mu' \in \Pi^+_\mathrm{tr}(\lambda)$ and by Remark \ref{oss I: caratterizzazione-pesi-banali} we get $\Omega(\lambda, \mu') = \Omega(\lambda)$. It follows then $\mu - \lambda \in \Omega(\lambda)$, i.e. $\mu \in \Pi^+_\mathrm{tr}(\lambda)$.
\end{proof}

\begin{lem} \label{lem II: ar-pari}
Let $\lambda \in \Lambda^+$ be such that $\alpha_r \not \in \Supp(\lambda)$, let $\mu \in \Pi^+(\lambda)$ and denote $\lambda -\mu = \sum_{i=1}^r a_i \alpha_i$. If $a_{r-1} = a_r$ is even, then $\mu \in \Pi^+_\mathrm{tr}(\lambda)$.
\end{lem}

\begin{proof}
By Corollary \ref{cor I: etichette}, the triviality of $\mu$ depends only on the equivalence class of the simple set $\{\lambda, \mu\}$. Hence we may replace the simple subset $\{\lambda, \mu\}$ with the equivalent simple subset $\{\lambda + \omega_{q(\lambda)}, \mu + \omega_{q(\lambda)}\}$, in particular we may assume that $\alpha_{q(\lambda)} \in \Supp(\mu)$.

We proceed by induction on $a_r$. Suppose that $a_r = 0$: then $\Supp_\Delta(\lambda - \mu)$ has all components of type ${\sf A}$ and the claim follows by Corollary \ref{cor I: differenza tipo A}. Suppose now that $a_{r-1} = a_r \geqslant 2$. Since $q(\lambda) < r$ and since $\mu$ is dominant, it must be $a_{q(\lambda)} \geqslant a_{q(\lambda) +1} \geqslant \ldots \geqslant a_{r-1} = a_r \geqslant 2$. Denote $p$ the maximum such that $a_p = 0$, or set $p=0$ if $a_i\neq 0$ for all $i$, and define
\[	\mu' = \mu + \sum_{i = p+1}^{q(\lambda)} \alpha_i + \sum_{i = q(\lambda) +1}^r 2\alpha_i =
			\mu - \varpi_p + \varpi_{p+1} - \varpi_{q(\lambda)} + \varpi_{q(\lambda) +1}.	\]
Notice that $\mu'$ is dominant: while $\alpha_{q(\lambda)} \in \Supp(\mu)$ by the assumption at beginning of the proof, if $p>0$ we have also $\alpha_p \in \Supp(\mu)$ since $\langle \mu, \alpha_p^\vee \rangle \geqslant \langle \lambda, \alpha_p^\vee \rangle + a_{p+1} > 0$. Therefore $\mu'\in \Pi^+(\lambda)$ and by construction we have $\mu < \mu'$. Hence by Lemma \ref{lem I: differenza radice positiva} i) and iii) we get $V(\lambda + \mu)\subset V(\lambda) \otimes  V(\mu')$, which implies $\mu-\lambda \in \Omega(\lambda,\mu')$.

Consider now $\mu'$ and denote $\lambda - \mu' = \sum a'_i \alpha_i$: then either $q(\lambda) = r-1$ and $\alpha_r \in \Supp(\mu')$ or $a'_{r-1} = a'_r = a_r -2$. It follows that $\mu' \in \Pi^+_\mathrm{tr}(\lambda)$, in the first case by Lemma \ref{lem II: mu-soddisfa-star} and in the second case by inductive hypothesis. Therefore by Remark \ref{oss I: caratterizzazione-pesi-banali} we get $\Omega(\lambda,\mu') = \Omega(\lambda)$ and it follows $\mu - \lambda \in \Omega(\lambda)$, i.e. $\mu \in \Pi^+_\mathrm{tr}(\lambda)$.
\end{proof}

\begin{lem} \label{lem II: ar-dispari}
Let $\lambda \in \Lambda^+$ be such that $\alpha_r \not \in \Supp(\lambda)$, let $\mu \in \Pi^+(\lambda)$ and denote $\lambda -\mu = \sum_{i=1}^r a_i \alpha_i$. If $a_r > 2\min\{r- q(\lambda), r-q(\mu)\}$, then $\mu \in \Pi^+_\mathrm{tr}(\lambda)$.
\end{lem}

\begin{proof}
By Corollary \ref{cor I: etichette}, the triviality of $\mu$ depends only on the equivalence class of the simple set $\{\lambda, \mu\}$. Hence we may replace the simple subset $\{\lambda, \mu\}$ with the equivalent simple subset $\{\lambda + \omega_{q(\lambda)}, \mu + \omega_{q(\lambda)}\}$, in particular we may assume that $q(\lambda) \leqslant q(\mu)$. Moreover by Lemma \ref{lem II: ar-pari} we may assume that $a_r$ is odd.

We proceed by induction on $r-q(\mu)$, the basis being the case $q(\mu)=r$ (Lemma \ref{lem II: mu-soddisfa-star}). Suppose that $q(\mu) < r$: since $\mu$ is dominant and since $a_r > 2(r-q(\mu))$, it must be
\[	a_{q(\lambda)} \geqslant a_{q(\lambda)+1} \geqslant \ldots \geqslant a_{q(\mu)} = a_{q(\mu)+1} = \ldots = a_r \geqslant 3. \]
Denote $p$ the maximum such that $a_p = 0$, or set $p=0$ otherwise, and define
\[
	\mu' = \mu + \sum_{i=p+1}^{q(\mu)} \alpha_i + \sum_{i=q(\mu)+1}^r 2\alpha_i = \mu - \varpi_p + \varpi_{p+1} - \varpi_{q(\mu)} + \varpi_{q(\mu)+1}.	\]
Notice that $\mu'$ is dominant: while $\alpha_{q(\mu)} \in \Supp(\mu)$ by definition, if $p>0$ we have also $\alpha_p \in \Supp(\mu)$ since $\langle \mu, \alpha_p^\vee \rangle \geqslant \langle \lambda, \alpha_p^\vee \rangle + a_{p+1} > 0$. Therefore $\mu'\in \Pi^+(\lambda)$ and by construction we have $\mu < \mu'$. Hence by Lemma \ref{lem I: differenza radice positiva} i), iii) we get $V(\lambda + \mu) \subset V(\lambda) \otimes V(\mu')$, which implies $\mu-\lambda \in \Omega(\lambda,\mu')$.

Consider now $\mu'$ and denote $\lambda - \mu' = \sum a'_i \alpha_i$. Suppose that $a'_r = 1$. Then $a_r=3$ and we get $q(\mu) = r-1$, since by hypothesis we have $a_r > 2(r-q(\mu))$ and $q(\mu) < r$: therefore $\alpha_r \in \Supp(\mu')$ and by Lemma \ref{lem II: mu-soddisfa-star} it follows $\mu' \in \Pi^+_\mathrm{tr}(\lambda)$. Otherwise we have $q(\mu') = q(\mu)+1 < r$ and $a'_r = a_r -2 > 2(r-q(\mu'))$: therefore $\mu' \in \Pi^+(\lambda)$ still satisfies the hypothesis of the lemma and by the inductive hypothesis we get $\mu' \in \Pi^+_\mathrm{tr}(\lambda)$. Therefore by Remark \ref{oss I: caratterizzazione-pesi-banali} it follows $\Omega(\lambda,\mu') = \Omega(\lambda)$ and we get $\mu - \lambda \in \Omega(\lambda)$, i.e. $\mu \in \Pi^+_\mathrm{tr}(\lambda)$.
\end{proof}

\section{Simple reduced subsets in the odd orthogonal case}

Let $\Pi \subset \Lambda^+$ be a simple subset. In this section we will define the \textit{reduction} of $\Pi$, which is the minimal simple subset $\Pi_{\mathrm{red}} \subset \Pi$ such that $X_\Pi$ and $X_{\Pi_{\mathrm{red}}}$ are equivariantly isomorphic. This subset is canonical, in the sense that if $\lambda \in \Pi$ is the maximal element, then the set of differences $\Pi_\mathrm{red} - \lambda$ depends only on $\Pi - \lambda$. If moreover $\Pi'\subset \Lambda^+$ is another simple subset such that 
\[    \xymatrix{
     \widetilde{X}_\lambda \ar@{->}[d] \ar@{->}[rr] & & X_{\Pi'} \ar@{->}[d]   \\
     X_{\Pi} \ar@{->}[rr] & & X_\lambda	}	\]
the reductions of $\Pi$ and $\Pi'$ will allow to characterize combinatorially the existence of an equivariant morphism $X_\Pi \to X_{\Pi'}$ which makes the diagram commute. In particular, it will follow a combinatorial criterion to establish whether two simple subsets give rise to isomorphic compactifications.

If $\lambda \in \Lambda^+$, denote $\Phi^+(\lambda) \subset \Phi^+$ the set of the postive roots non-orthogonal to $\lambda$ and consider the following degeneration of the dominance order:
\[	\nu \leqslant^\lambda \mu \qquad \text{ if and only if } \qquad \mu - \nu \in \mathbb{N} \Phi^+(\lambda).		\]
Notice that $\leqslant^\lambda$ depends only on $\Supp(\lambda)$ and that it coincides with the usual dominance order if $\lambda$ is a regular weight. The partial order $\leqslant^\lambda$ was studied in the general semisimple case by Gandini and Ruzzi in \cite{GR}, where it is used to characterize the normality of a simple linear compactification of a semisimple group. In particular, there are proved the following properties.

\begin{prop}[see {\cite[Prop. 2.1 and Cor. 2.2]{GR}}] \label{prop III: lambda dominanza in natura}
Let $\lambda \in \Lambda^+$.
\begin{itemize}
\item[i)] If $\mu \in \Pi(\lambda)$, then $\mu \leqslant^\lambda \lambda$.
\item[ii)] Let $\lambda_1, \ldots, \lambda_n \in \Pi^+(\lambda)$. If $\mu, \nu \in \Lambda^+$ are such that $V(\nu) \subset V(\mu) \otimes V(\lambda_1) \otimes \ldots \otimes V(\lambda_n)$, then $\nu \leqslant^\lambda \mu + n \lambda$.
\end{itemize}
\end{prop}

Suppose that $\Pi \subset \Lambda^+$ is simple with maximal element $\lambda$ and define the \textit{reduction} of $\Pi$ as follows:
\[	\Pi_{\mathrm{red}} \doteq \{ \mu \in \Pi \smallsetminus \Pi^+_{\mathrm{tr}}(\lambda) \, : \,  \mu \text{ is maximal w.r.t. } \leqslant^\lambda \} \cup \{\lambda\}.		\]
If $\Pi = \Pi_{\mathrm{red}}$, then we say that $\Pi$ is \textit{reduced}. For instance, if $\alpha_r \not \in \Supp(\lambda)$ and $\Pi = \Pi^+(\lambda)$, then we have $\Pi_\mathrm{red} = \{\lambda, \lambda^\mathrm{lb}\}$: this is a consequence of Proposition \ref{prop III: lambda dominanza in natura} together with Theorem \ref{teo I: BGMR}.

Let $\lambda \in \Lambda^+$ and denote $\Xi(\lambda) \subset \mathbb{N} \Delta$ the semigroup of the elements $\tau = \sum a_i \alpha_i$ which satisfy the following inequalities:
\begin{itemize}
	\item[\textbf{($\lambda$-C1)}] If $p\geqslant 1$ is the minimum such that $\alpha_p \in \Supp(\lambda)$, then $a_1 \leqslant a_2 \leqslant \ldots \leqslant a_{p}$.
	\item[\textbf{($\lambda$-C2)}] If $s < t$ are such that $\alpha_s, \alpha_t \in \Supp(\lambda)$ and $\alpha_i \not \in \Supp(\lambda)$ for every $s < i < t$, then $\sum_{i=s}^{t-1} |a_{i+1} - a_i| \leqslant a_s + a_t$.
	\item[\textbf{($\lambda$-C3)}] Let $q \leqslant r$ be the maximum such that $\alpha_q \in \Supp(\lambda)$. If $q<r$, then $a_r$ is even and $2 \sum_{i \in I_q} (a_{i+1}-a_i) \leqslant a_r$, where $I_q = \{ i <r \, : \, i\geqslant q \text{ and } a_i < a_{i+1}\}$.
\end{itemize}

\begin{oss}	\label{oss III: X(lambda) e radici}
If $\theta \in \Phi^+_l$ is a long root, then they exist $j,k \in \mathbb N$ with $j < k < r$ such that $\theta = \sum_{i=j+1}^{k} \alpha_i$ or $\theta = \sum_{i=j+1}^{k} \alpha_i + 2 \sum_{i=k+1}^r \alpha_i$, while if $\theta \in \Phi^+_s$ is a short root, then there exists $j < r$ such that $\theta = \sum_{i=j+1}^{r} \alpha_i$. Let now $\lambda \in \Lambda^+$: then by the description above it follows that $\Phi^+_l(\lambda) \cup 2\Phi^+_s(\lambda) \subset \Xi(\lambda)$, while $\Phi^+_s(\lambda) \subset \Xi(\lambda)$ if and only if $\alpha_r \in \Supp(\lambda)$. Notice that we always have $\Phi^+_s(\lambda) + \Phi^+_s(\lambda) \subset \Xi(\lambda)$: indeed $\Phi^+_s(\lambda) + \Phi^+_s(\lambda) \subset \Phi^+_l(\lambda) \cup 2\Phi^+_s(\lambda)$.
\end{oss}

\begin{lem} \label{lem III: X(lambda)}
Let $\lambda \in \Lambda^+$ and $\tau \in \Xi(\lambda)$ be non-zero. There exists $\theta \in \Phi^+(\lambda)$ such that $\tau - \theta \in \Xi(\lambda)$ and $\lambda + \tau^- + \theta \in \Lambda^+$. If moreover $\alpha_r \not\in \Supp(\lambda)$, then it is possible to choose $\theta \in \Phi^+_l(\lambda)$.
\end{lem}

\begin{proof}
Denote $p$ the minimum such that $\alpha_p \in \Supp(\lambda)$ and $q$ the maximum such that $\alpha_q \in \Supp(\lambda)$ and let $\tau = \tau^+ - \tau^- = \sum a_i \alpha_i \in \Xi(\lambda)$. Denote $s_0 \geqslant 0$ the minimum such that $\alpha_{s_0} \in \Supp(\lambda) \cap \Supp_\Delta(\tau)$ and notice that $a_i \leqslant a_{i+1}$ for every $i < s_0$. This follows by \textbf{($\lambda$-C1)} if $s_0 = p$, whereas if $s_0 > p$ then \textbf{($\lambda$-C2)} implies $\sum_{i=1}^{s_0-1} |a_{i+1} - a_i| = a_{s_0} = \sum_{i=1}^{s_0-1} (a_{i+1} - a_i)$, hence $a_{i+1} - a_i \geq 0$ for every $i < s_0$. Denote $j < s_0$ the maximum such that $a_{j+1} \neq 0$: then $0 < a_{j+1} \leqslant \ldots \leqslant a_{s_0}$. Notice that, if $j>0$, then $\alpha_j \in \Supp(\tau^-)$: indeed $\langle \tau^-, \alpha_j^\vee \rangle = \langle \tau^+, \alpha_j^\vee \rangle + a_{j+1} >0$. In order to construct the root $\theta$, we distinguish the following cases:
\begin{itemize}
	\item[\textit{Case 1.}] $s_0 < q$;
	\item[\textit{Case 2.}] $s_0 = q = r$;
	\item[\textit{Case 3.}] $s_0 = q$ and $a_r = 0$;
	\item[\textit{Case 4.}] $s_0 = q < r$ and $a_r \neq 0$.
\end{itemize}

\textit{Case 1.} Suppose that $s_0 < q$. Denote $t_0 > s_0$ the minimum such that $\alpha_{t_0} \in \Supp(\lambda)$ and define $k$ as follows:
\[
	k = 
		\left\{ \begin{array}{ll}
			t_0-1 & \text{ if } a_{s_0} \leqslant a_{s_0 + 1} \leqslant \ldots \leqslant a_{t_0} \\
			\max\{ i < t_0 \, : \, a_i > a_{i+1} \} & \text{ otherwise }
		\end{array} \right.
\]
Therefore $j < s_0 \leqslant k < t_0$. Set $\theta = \sum_{i=j+1}^k \alpha_i$ and denote $\tau' = \tau - \theta$: then by construction $\tau' \in \mathbb{N}\Delta$ and $\theta \in \Phi^+_l(\lambda)$, we claim that $\tau' \in \Xi(\lambda)$. Notice that $\tau'$ satisfies \textbf{($\lambda$-C1)} and \textbf{($\lambda$-C3)} since $\tau$ is so, therefore we only need to show that $\tau'$ satisfies \textbf{($\lambda$-C2)}. Denote $\tau' = \sum_{i=1}^r a'_i \alpha_i$ and suppose that $s<t$ are such that $\alpha_s, \alpha_t \in \Supp(\lambda)$ and $\alpha_i \not \in \Supp(\lambda)$ for every $s < i < t$. Since otherwise $a'_i = a_i$ for every $i$ with $s \leqslant i \leqslant t$, we may assume that either $s\leqslant j < t$ or $s\leqslant k<t$: then $a'_s + a'_t = a_s + a_t -1$ and
\[	|a'_{i+1} - a'_i| =	\left\{ \begin{array}{ll}
				|a_{i+1} - a_i| & \text{ if } i \in \{s, \ldots, t-1 \} \smallsetminus \{j,k\} \\
				|a_{i+1} - a_i| - 1 & \text{ if } i = j \text{ or } i = k 	\end{array} \right.	\]
By construction, $\{s, \ldots, t-1 \}$ cannot contain both $j$ and $k$. Therefore we get
\[	\sum_{i=s}^{t-1} |a'_{i+1} - a'_i| = \sum_{i=s}^{t-1} |a_{i+1}-a_i| -1 \leqslant a_s + a_t -1 = a'_s + a'_t	\]
and $\tau'$ satisfies \textbf{($\lambda$-C2)}.

To show that that $\lambda + \tau^- + \theta \in \Lambda^+$, notice that $\lambda + \tau^- + \theta = \lambda + \tau^- - \varpi_j + \varpi_{j+1} + \varpi_k - \varpi_{k+1}$. Suppose that $\alpha_{k+1} \not \in \Supp(\lambda)$: then by the definition of $k$ we have $a_k > a_{k+1} \leqslant a_{k+2}$ and it follows $\alpha_{k+1}\in \Supp(\tau^-)$ since
\[	\langle \tau^-, \alpha_{k+1}^\vee \rangle = \langle \tau^+, \alpha_{k+1}^\vee \rangle + a_k - 2 a_{k+1} + a_{k+2} > 0.	\]

\textit{Case 2.} Suppose that $s_0 = q = r$, denote $\theta = \sum_{i=j+1}^r \alpha_i$ and set $\tau' = \tau - \theta$. Then $\theta \in \Phi^+_s(\lambda)$ and by the definition of $j$ we have $\Supp_\Delta(\tau) = \{\alpha_{j+1}, \ldots, \alpha_r\}$, $\Supp(\lambda) \cap \Supp_\Delta(\tau) = \{\alpha_r\}$ and $0 < a_{j+1} \leqslant \ldots \leqslant a_r$. If moreover $\tau' = \sum a'_i \alpha_i$, then we still have $\Supp_\Delta(\tau') \subset \{\alpha_{j+1}, \ldots, \alpha_r\}$, $\Supp(\lambda) \cap \Supp_\Delta(\tau) = \{\alpha_r\}$ and $0 \leqslant a'_{j+1} \leqslant \ldots \leqslant a'_r$: therefore $\tau' \in \Xi(\lambda)$ and $\tau^- + \theta = \tau^- - \varpi_j + \varpi_{j+1} \in \Lambda^+$.

\textit{Case 3.} Suppose that $s_0 = q$ and that $a_r = 0$. Since $a_{s_0} >0$, it must be $q < r$, hence \textbf{($\lambda$-C3)} implies $a_q \geqslant a_{q+1} \geqslant \ldots \geqslant a_r = 0$. Denote $k \geqslant q$ the maximum such that $a_k > 0$: then by the definition of $j$ we get $\Supp_\Delta(\tau) = \{\alpha_{j+1}, \ldots, \alpha_k\}$, $\Supp(\lambda) \cap \Supp_\Delta(\tau) = \{\alpha_q\}$ and
\[	0 < a_{j+1} \leqslant \ldots \leqslant a_q \geqslant \ldots \geqslant a_k > 0.	\]
Set $\theta = \sum_{i=j+1}^k \alpha_i$ and $\tau' = \tau - \theta$: then by construction $\theta  \in \Phi^+_l(\lambda)$ and $\tau' \in \mathbb{N}\Delta$. If moreover $\tau' = \sum a'_i \alpha_i$, then we still have $\Supp_\Delta(\tau') \subset \{\alpha_{j+1}, \ldots, \alpha_k\}$, $\Supp(\lambda) \cap \Supp_\Delta(\tau') = \{\alpha_q\}$ and $0 \leqslant a'_{j+1} \leqslant \ldots \leqslant a'_q \geqslant \ldots \geqslant a'_k \geqslant 0$: therefore $\tau' \in \Xi(\lambda)$. Consider now $\tau^- + \theta = \tau^- - \varpi_j + \varpi_{j+1} + \varpi_k - \varpi_{k+1}$. Then we have $\langle \tau^-, \alpha_{k+1}^\vee \rangle = \langle \tau^+, \alpha_{k+1}^\vee \rangle + a_k > 0$, therefore $\tau^- + \theta \in \Lambda^+$.

\textit{Case 4.} Suppose that $s_0 = q < r$ and that $a_r \neq 0$. If $a_i = 0$ for some $i > q$, then \textbf{($\lambda$-C3)} implies that $a_r = 0$, hence it must be $a_i \neq 0$ for every $q \leqslant i \leqslant r$. Therefore we have $\Supp_\Delta(\tau) = \{\alpha_{j+1}, \ldots, \alpha_r\}$ and $\Supp(\lambda) \cap \Supp_\Delta(\tau) = \{\alpha_q\}$. Denote $I_q = \{i < r \, : \, i\geqslant q \text{ and } a_i < a_{i+1} \}$ and define $k$ as follows:
\[	k = \left\{ \begin{array}{ll}
			q & \text{ if } I_q = \varnothing \\
			\min\{ i \geqslant q \, : \, a_i < a_{i+1} \} & \text{ if } I_q \neq \varnothing \end{array} \right.	\]
Therefore we have $j < q \leqslant k < r$. Notice that $a_i \geqslant 2$ for every $k < i \leqslant r$: indeed $a_r \geqslant 2$ by the definition of $\Xi(\lambda)$, whereas if $I_q \neq \varnothing$ and $a_i = 1$ for some $i > q$ then \textbf{($\lambda$-C3)} implies $a_r = 2$ and $I_q = \{k\}$. Therefore, if we set $\theta = \sum_{i=j+1}^k \alpha_i + 2\sum_{i=k+1}^r \alpha_i$ and $\tau' = \tau - \theta$, then we have $\theta \in \Phi^+_l(\lambda)$ and $\tau' \in \mathbb{N}\Delta$. We claim that $\tau' \in \Xi(\lambda)$. Since $a_{j+1} \leqslant \ldots \leqslant a_q$, we have that $\tau'$ satisfies \textbf{($\lambda$-C1)} and \textbf{($\lambda$-C2)} as a direct consequence of the fact that these conditions are satisfied by $\tau$. To show that $\tau'$ satisfies \textbf{($\lambda$-C3)}, denote $\tau' = \sum_{i=1}^r a'_i \alpha_i$ and set $I'_q = \{i\geqslant q \, : \, a'_i < a'_{i+1} \}$. If $i \geqslant q$, notice that we have $a'_{i+1} - a'_i = a_{i+1} - a_i$ unless $i = k$, in which case $a'_{k+1} - a'_k = a_{k+1} - a_k -1$. Hence we get that
\[	2 \sum_{i \in I'_q} (a'_{i+1} - a'_i) = 2 \sum_{i \in I_q} (a_{i+1} - a_i) - 2 \leqslant a_r - 2 = a'_r,	\]
therefore $\tau' \in \Xi(\lambda)$. To show that $\lambda + \tau^- + \theta \in \Lambda^+$, since $a_r \neq 0$, notice that we have $\tau^- + \theta = \tau^- - \varpi_j + \varpi_{j+1} - \varpi_k + \varpi_{k+1}$. If $k \neq q$, then by its definition we have $a_{k-1} \geqslant a_k < a_{k+1}$. Therefore if $\alpha_k \not \in \Supp(\lambda)$ it follows that
\[	\langle \tau^-, \alpha_k^\vee \rangle = \langle \tau^+, \alpha_k^\vee \rangle + a_{k-1} + a_{k+1} - 2 a_k > 0: \]
hence $\alpha_k \in \Supp(\tau^-)$ and $\lambda + \tau^- + \theta \in \Lambda^+$.
\end{proof}

\begin{cor} \label{cor III: X(lambda)}
Let $\lambda \in \Lambda^+$.
\begin{itemize}
	\item[i)] If $\alpha_r \in \Supp(\lambda)$, then $\Xi(\lambda) = \mathbb{N} \Phi^+(\lambda)$.
	\item[ii)] If $\alpha_r \not \in \Supp(\lambda)$, then $\Xi(\lambda) = \mathbb{N} \Phi^+_l(\lambda)$.
\end{itemize}
\end{cor}

\begin{proof}
By Remark \ref{oss III: X(lambda) e radici} we have $\Phi^+_l(\lambda) \subset \Xi(\lambda)$, whereas $\Phi^+_s(\lambda) \subset \Xi(\lambda)$ if and only if $\alpha_r \in \Supp(\lambda)$. Let $\tau \in \Xi(\lambda)$ and denote $\tau = \sum_{i=1}^r a_i \alpha_i$. If $\tau \neq 0$, then by Lemma \ref{lem III: X(lambda)} there exists $\tau' \in \Xi(\lambda)$ with $\tau - \tau' \in \Phi^+(\lambda)$, and if $\alpha_r \not \in \Supp(\lambda)$ we may assume $\tau - \tau' \in \Phi^+_l(\lambda)$. Therefore the claim follows proceeding by induction on $\sum_{i=1}^r a_i$.
\end{proof}

\begin{teo} \label{teo III: lambda-confr -> morf}
Let $\mu, \nu \in \Pi^+(\lambda) \smallsetminus \Pi^+_\mathrm{tr}(\lambda)$. Then $\Omega(\lambda, \nu) \subset \Omega(\lambda,\mu)$ if and only if $\nu \leqslant^\lambda \mu$.
\end{teo}

\begin{proof}
Since otherwise $\Pi^+_\mathrm{tr}(\lambda) = \Pi^+(\lambda)$, it must be $\alpha_r \not \in \Supp(\lambda)$. Suppose that $\Omega(\lambda, \nu) \subset \Omega(\lambda,\mu)$. In particular we have $\nu - \lambda \in \Omega(\lambda,\mu)$, so it exists $k \leqslant n$ such that $V((n-1)\lambda + \nu) \subset V(\lambda)^{\otimes k} \otimes V(\mu)^{\otimes n-k}$. On the other hand $\nu \not \in \Pi^+_\mathrm{tr}(\lambda)$ so by Remark \ref{oss I: caratterizzazione-pesi-banali} i) it must be $k < n$ and by Proposition \ref{prop III: lambda dominanza in natura} ii) it follows $\nu \leqslant ^\lambda \mu$.

Suppose conversely that $\nu \leqslant^\lambda\mu$, we will show that $\Omega(\lambda,\nu) \subset \Omega(\lambda,\mu)$ proceeding by induction on the difference $\mu - \nu$. Denote $\lambda - \mu = \sum_{i=1}^r m_i \alpha_i$, $\lambda - \nu = \sum_{i=1}^r n_i \alpha_i$ and set $a_i = m_i - n_i$. By Theorem \ref{teo II: classificazione pesi banali B} $m_r$ and $n_r$ are both odd integers, so that $a_r$ is even. Denote $\mu - \nu = \sum a_i \alpha_i = \theta_1 + \ldots + \theta_n$ with $\theta_i \in \Phi^+(\lambda)$ and let $k$ be the number of short roots which occur in the set $\{\theta_1, \ldots, \theta_n\}$. Denote $\theta_i = \sum b^i_j \alpha_j$ and notice that $\theta_i$ is short if and only if $b^i_r$ is odd. Since $a_r$ is even, $k$ is even as well: by Remark \ref{oss III: X(lambda) e radici} together with Corollary \ref{cor III: X(lambda)} it follows then $\mu - \nu \in \mathbb{N} \Phi^+_l(\lambda) = \Xi(\lambda)$. Hence by Lemma \ref{lem III: X(lambda)} together with Corollary \ref{cor III: X(lambda)} there exists $\theta \in \Phi^+_l(\lambda)$ such that $\mu - \nu - \theta \in \mathbb{N} \Phi^+_l(\lambda)$ and $\lambda + \nu + \theta \in \Lambda^+$.

Denote $\mu' = \lambda + \mu$, $\nu' = \lambda + \nu$ and $\pi = \lambda + \nu + \theta$. Then $\{\lambda, \mu\} \sim \{2\lambda, \mu'\}$ and $\{\lambda, \nu\} \sim \{2\lambda, \nu'\}$ are equivalent simple subsets, so that by Corollary \ref{cor I: etichette} i) it follows $\Omega(2\lambda, \mu') = \Omega(\lambda, \mu)$ and $\Omega(\lambda,\nu) = \Omega(2\lambda, \nu')$. Moreover by Lemma \ref{lem I: differenza radice positiva} i) it follows that $V(2\lambda + \nu') \subset V(2\lambda) \otimes V(\pi)$, hence $\nu' - 2\lambda \in \Omega(2\lambda, \pi)$ and by Corollary \ref{cor I: criterio semigruppi} we get $\Omega(2\lambda, \nu') \subset \Omega(2\lambda,\pi)$. Consider now the weights $\mu', \pi \in \Pi^+(2\lambda)$: then we have $\pi \leqslant^\lambda \mu'$ and $\mu' - \pi < \mu - \nu$, hence by the inductive hypothesis it follows $\Omega(2\lambda, \pi) \subset \Omega(2\lambda, \mu')$.
Therefore we get
\[	\Omega(\lambda,\nu) = \Omega(2\lambda, \nu') \subset \Omega(2\lambda,\pi) \subset \Omega(2\lambda, \mu') = \Omega(\lambda, \mu). 	\qedhere	\]
\end{proof}

\begin{cor}\label{cor III: morfismi gen}
Let $\Pi, \Pi'\subset \Lambda^+$ be simple with maximal elements resp. $\lambda, \lambda'$ and assume that $\Supp(\lambda) = \Supp(\lambda')$.
\begin{itemize}
	\item[i)] There exists an equivariant morphism $X_\Pi \to X_{\Pi'}$ if and only if for every $\mu' \in \Pi'_\mathrm{red}$ there exists $\mu \in \Pi_\mathrm{red}$ such that $\mu' - \lambda' \leqslant^\lambda \mu - \lambda$.
	\item[ii)] The varieties $X_\Pi$ and $X_{\Pi'}$ are equivariantly isomorphic if and only if $\Pi_\mathrm{red} \sim \Pi'_\mathrm{red}$.
\end{itemize}
\end{cor}

\begin{proof}
i) Suppose that $X_\Pi$ dominates $X_{\Pi'}$ and let $\mu' \in \Pi'_\mathrm{red} \smallsetminus \{\lambda'\}$. By Corollary \ref{cor I: etichette} i) we have $\Omega(\lambda) = \Omega(\lambda')$, while by Proposition \ref{prop I: criterio tensoriale} they exist $\mu_1, \ldots, \mu_n \in \Pi$ such that $V(\mu' - \lambda' + n\lambda) \subset V(\mu_1) \otimes \ldots \otimes V(\mu_n)$. Since $\mu' \in \Pi' \smallsetminus \Pi^+_{\mathrm{tr}}(\lambda')$, we have $\mu' - \lambda' \not \in \Omega(\lambda') = \Omega(\lambda)$, hence some $\mu_i$ is not trivial, say $\mu_1$, and Proposition \ref{prop III: lambda dominanza in natura} ii) implies $\mu' - \lambda' \leqslant^\lambda \mu_1 - \lambda$. Therefore if $\mu \in \Pi_\mathrm{red}$ is any weight such that $\mu_1 \leqslant^\lambda \mu$, we get $\mu' - \lambda \leqslant^\lambda \mu - \lambda$.

Suppose conversely that for every $\mu' \in \Pi'_\mathrm{red}$ there exists $\mu \in \Pi_\mathrm{red}$ such that $\mu' - \lambda' \leqslant^\lambda \mu - \lambda$ and let $\nu' \in \Pi'$. If $\nu' \in \Pi^+_\mathrm{tr}(\lambda')$, then $\nu'-\lambda' \in \Omega(\lambda') = \Omega(\lambda)$ by Proposition \ref{prop I: morfismi X_lambda} together with Remark \ref{oss I: caratterizzazione-pesi-banali}. Suppose that $\nu' \in \Pi' \smallsetminus \Pi^+_\mathrm{tr}(\lambda')$ and let $\mu' \in \Pi'_\mathrm{red}$ be such that $\nu' \leqslant^\lambda \mu'$: then by Theorem \ref{teo III: lambda-confr -> morf} it follows that $\Omega(\lambda',\nu') \subset \Omega(\lambda',\mu')$. Let $\mu \in \Pi_\mathrm{red}$ be such that $\mu' - \lambda' \leqslant^\lambda \mu - \lambda$ and denote $\lambda'' = \lambda + \lambda'$: since $\mu$ and $\mu'$ are non-trivial, it follows that $\lambda + \mu', \lambda' + \mu \in \Pi^+(\lambda'') \smallsetminus \Pi^+_\mathrm{tr}(\lambda'')$. Then $\{\lambda,\mu\} \sim \{\lambda'',\lambda'+\mu\}$ and $\{\lambda', \mu'\} \sim \{\lambda'', \lambda+\mu'\}$ are equivalent simple sets, so by Corollary \ref{cor I: etichette} i) we have $\Omega(\lambda,\mu) = \Omega(\lambda'',\lambda'+\mu)$ and $\Omega(\lambda',\mu') = \Omega(\lambda'',\lambda+\mu')$. Moreover by construction we have $\lambda +\mu' \leqslant^\lambda \lambda' + \mu$, hence $\Omega(\lambda'',\lambda+\mu') \subset \Omega(\lambda'',\lambda'+\mu)$ by Theorem \ref{teo III: lambda-confr -> morf}. Finally by the definition of $\Omega(\Pi)$ we have $\Omega(\lambda,\mu) \subset \Omega(\Pi)$, so it follows
\[	\Omega(\lambda',\nu') \subset \Omega(\lambda',\mu') = \Omega(\lambda'',\lambda+\mu') \subset \Omega(\lambda'',\lambda'+\mu) = \Omega(\lambda,\mu) \subset \Omega(\Pi).	\]
Therefore we have shown that $\nu' - \lambda' \in \Omega(\Pi)$ for every $\nu' \in \Pi'$ and the claim follows by Corollary \ref{cor I: criterio semigruppi}.

ii) If $\Pi_\mathrm{red} \sim \Pi'_\mathrm{red}$ then the claim follows by Corollary \ref{cor I: etichette} i). Assume that $X_\Pi \simeq X_{\Pi'}$ and let $\mu \in \Pi_{\mathrm{red}} \smallsetminus \{\lambda\}$. Then by i) there exists $\mu' \in \Pi'_\mathrm{red}$ such that $\mu - \lambda \leqslant^\lambda \mu' - \lambda'$ and similarly there exists $\mu_1 \in \Pi_\mathrm{red}$ such that $\mu' - \lambda' \leqslant^\lambda \mu_1 - \lambda$. Therefore $\mu + \lambda' \leqslant^\lambda \mu' + \lambda \leqslant^\lambda \mu_1 + \lambda'$ and we get $\mu \leqslant^\lambda \mu_1$. On the other hand by the definition of $\Pi_{\mathrm{red}}$ we have that $\mu, \mu_1 \in \Pi \smallsetminus \Pi^+_{\mathrm{tr}}$ are maximal w.r.t. $\leqslant^\lambda$, hence it follows that $\mu = \mu_1$ and $\mu - \lambda = \mu' - \lambda'$. Therefore for every $\mu \in \Pi_{\mathrm{red}}$ there exists $\mu' \in \Pi'_\mathrm{red}$ such that $\mu - \lambda = \mu' - \lambda'$, and an analogous argument for $\Pi'$ shows that $\Pi_\mathrm{red} - \lambda = \Pi'_\mathrm{red} - \lambda$. On the other hand, since the closed orbits of $X_\Pi$ and $X_{\Pi'}$ are isomorphic, Proposition \ref{prop I: timashev} ii) implies that $\Supp(\lambda) = \Supp(\lambda')$, therefore we get $\Pi_\mathrm{red} \sim \Pi'_\mathrm{red}$.
\end{proof}

\section{Examples: simple linear compactifications of $\mathrm{SO}(7)$ and $\mathrm{SO}(9)$}

If $I \subset \Delta$, set $X_I = X_\lambda$ and $\widetilde{X}_I \to X_I$ the normalization, where $\lambda \in \Lambda^+$ is such that $\Supp(\lambda) = I$: by Proposition \ref{prop I: morfismi X_lambda} these varieties are well defined and they only depend on $I$. Consider the set
\[T(I) = \{\text{simple linear compactifications $\mathrm{SO}(2r+1) \hookrightarrow X$ such that $\widetilde{X}_I \to X \to X_I$}\}. \]
In other words, following the discussion after Proposition \ref{prop I: timashev}, $T(I)$ is the set of the compactifications $X_\Pi$ such that $\Pi \subset \Lambda^+$ is a simple subset whose maximal element has support $I$. We regard $T(I)$ as a partially ordered set, where the order is defined as follows: $X' \leqslant X$ if there exists an equivariant morphism  $X \to X'$. We also denote
\[T(I,2) = \{X \in T(I) \, : \, X \not \simeq X_I \text{ and } X \simeq X_\Pi \text{ for some } \Pi \subset \Lambda^+ \text{ with } \mathrm{card}(\Pi) = 2\}.\]

Denote $\leqslant^I$ the partial order on $\mathbb N\Delta$ defined as follows:
\[
\text{If } \theta_1, \theta_2 \in \mathbb N\Delta \text{, then } \theta_1 \leqslant^I \theta_2 \text{ if and only if } \theta_2 - \theta_1 \in \mathbb{N}\Delta \smallsetminus \mathbb{N} [\Delta \smallsetminus I].
\]
This coincides with the partial order $\leqslant^\lambda$ defined in Section 3, where $\lambda$ is any weight such that $\Supp(\lambda) = I$. Following Remark \ref{oss I: etichette}, Corollary \ref{cor II: descrizione anello delle coordinate B} and Corollary \ref{cor III: morfismi gen}, $\big(T(I,2), \leqslant\big)$ is identified with the partially ordered set $\big( \mathcal T(I,2), \leqslant^I \big)$, where $\mathcal T(I,2) \subset \mathbb{N}\Delta$ is defined as follows:
\[	\mathcal{T}(I,2) = \left\{ \theta = \sum_{i=1}^r a_i \alpha_i  \in \mathbb{N}\Delta \, : \,
				 \Supp(\theta^+) \subset I, a_r \text{ is odd and } a_r < 2\min\{r-l(\theta),r-q(I) \}\right\}
\]
where $q(I) \leqslant r$ is the maximum such that $\alpha_{q(I)} \in I$ and where $l(\theta)\leqslant r$ denotes the minimum such that $a_i = a_r$ for every $i \geqslant l(\theta)$. Following Theorem \ref{teo I: BGMR}, notice that $\mathcal T(I,2)$ possesses a unique maximal element w.r.t. $\leqslant^I$, namely $\theta_I = \sum_{i=q(I)}^r \alpha_i$: this is the element corresponding to the normalization $\widetilde X_I$ and it coincides with the difference $\lambda - \lambda^{\mathrm{lb}}$, where $\lambda \in \Lambda^+$ is any dominant weight with $\Supp(\lambda) = I$ and where $\lambda^{\mathrm{lb}}$ is the little brother of $\lambda$ (see Definition \ref{def: twin}).

Denote $\mathcal P(\mathbb N \Delta)$ the power set of $\mathbb N \Delta$ and extend $\leqslant^I$ to a partial order relation on $\mathcal P(\mathbb N \Delta)$ as follows:
\[
\text{If } A, A' \subset \mathbb N \Delta, \text{ then } A \leqslant^I A' \text{ if and only if } \forall \theta \in A \quad \exists \theta' \in A' \, : \, \theta \leqslant^I \theta'.
\]
Following Corollary \ref{cor III: morfismi gen}, every element of $T(I)$ is identified with a subset of $T(I,2)$ and we may identify $\big( T(I), \leqslant \big)$ with the partially ordered set $\big( \mathcal T(I), \leqslant^I \big)$, where $\mathcal T(I) \subset \mathcal P(\mathcal T(I,2))$ is defined as follows:
\[
	\mathcal T(I) = \{ A \subset \mathcal T(I,2) \, : \, A \text{ contains no comparable elements w.r.t. } \leqslant^I \}.
\]

In the following tables we represent the poset $\big(\mathcal T(I,2), \leqslant^I\big)$ for $\mathrm{SO}(7)$ and for $\mathrm{SO}(9)$. If $\alpha_r \in I$ then by Theorem \ref{teo I: BGMR} it follows that $\mathcal{T}(I,2) = \varnothing$, while if $I = \{\alpha_{r-1}\}$ it follows by Remark \ref{oss II: omega_r-1} that $\mathcal{T}(I,2) = \{ \alpha_{r-1}+\alpha_r	\}$. Therefore we will assume that $\alpha_r \not \in I$ and that $I \neq \{\alpha_{r-1}\}$. We represent an element $\sum a_i \alpha_i \in \mathcal{T}(I,2)$ as the vector $(a_1, \ldots, a_r)$ and we connect two elements $\theta_1, \theta_2 \in \mathcal{T}(I,2)$  with an arrow $\theta_1 \rightarrow \theta_2$ if and only if $\theta_1 \leqslant^I \theta_2$ and $\theta_2 \in \mathcal{T}(I,2)$ is minimal with this property.

\vspace{0.1cm}

{\scriptsize

\begin{table}[h]
\caption{The poset $\big(\mathcal T(I,2), \leqslant^I \big)$ for $G = \mathrm{SO}(7)$, $I = \{\alpha_1, \alpha_2 \}$.}
\begin{center}
\vspace{-0.4cm}
\[
\xymatrix
    {
         (0,1,1) \ar@{->}[r] & (1,1,1) \ar@{->}[r] & (2,1,1) \ar@{->}[r] & (3,1,1) \ar@{->}[r] & (4,1,1) \ar@{->}[r] & \ldots
    }
\]
\end{center}
\end{table}

\begin{table}[h]
\caption{The poset $\big(\mathcal T(I,2), \leqslant^I \big)$ for $G = \mathrm{SO}(7)$, $I = \{\alpha_1 \}$.}
\begin{center}
\vspace{-0.5cm}
\[
\xymatrix@R=3pt@C=1pt{
		& & & & & (3,3,3) \\
        (1,1,1) \ar@{->}[rrr] & & & (2,1,1) \ar@{->}[rrr] \ar@{->}[urr] & & & (3,1,1) \ar@{->}[rrr] & & & (4,1,1) \ar@{->}[rrr] & & & (5,1,1) \ar@{->}[rrr] & & & \ldots
        }
\]        
\end{center}
\end{table}

{\scriptsize

\begin{table}[h]
\caption{The poset $\big(\mathcal T(I,2), \leqslant^I \big)$ for $G = \mathrm{SO}(9)$, $I = \{\alpha_1, \alpha_2, \alpha_3 \}$.}
\begin{center}
\vspace{-0.4cm}
\[
\xymatrix@R=3pt@C=1pt{
        (0,0,1,1) \ar@{->}[rrr] \ar@{->}[dd] & & & (0,1,1,1) \ar@{->}[rrr] \ar@{->}[dd] & & & (0,2,1,1) \ar@{->}[rrr] \ar@{->}[dd] & & & (0,3,1,1) \ar@{->}[rrr] \ar@{->}[dd] & & & (0,4,1,1) \ar@{->}[rr] \ar@{->}[dd] & & \ldots \\
        & \\
        (1,0,1,1) \ar@{->}[rrr] \ar@{->}[dd] & & & (1,1,1,1) \ar@{->}[rrr] \ar@{->}[dd] & & & (1,2,1,1) \ar@{->}[rrr] \ar@{->}[dd] & & & (1,3,1,1) \ar@{->}[rrr] \ar@{->}[dd] & & & (1,4,1,1) \ar@{->}[rr] \ar@{->}[dd] & & \ldots \\
        & \\
        (2,0,1,1) \ar@{->}[rrr] \ar@{->}[dd] & & & (2,1,1,1) \ar@{->}[rrr] \ar@{->}[dd] & & & (2,2,1,1) \ar@{->}[rrr] \ar@{->}[dd] & & & (2,3,1,1) \ar@{->}[rrr] \ar@{->}[dd] & & & (2,4,1,1) \ar@{->}[rr] \ar@{->}[dd] & & \ldots \\
        & \\
\vdots & & & \vdots & & & \vdots & & & \vdots & & & \vdots
    }
\]
\end{center}
\end{table}

\begin{table}[h]
\caption{The poset $\big(\mathcal T(I,2), \leqslant^I \big)$ for $G = \mathrm{SO}(9)$, $I = \{\alpha_1, \alpha_3 \}$.}
\begin{center}
\vspace{-0.6cm}
\[
\xymatrix@R=3pt@C=1pt{
        (0,0,1,1) \ar@{->}[rrr] \ar@{->}[rrrdd] & & & (1,0,1,1) \ar@{->}[rrr] \ar@{->}[rrrdd] & & & (2,0,1,1) \ar@{->}[rrr] \ar@{->}[rrrdd] & & & (3,0,1,1) \ar@{->}[rrr] \ar@{->}[rrrdd] & & & (4,0,1,1) \ar@{->}[rrr] \ar@{->}[rrrdd] & & & (5,0,1,1) \ar@{->}[rrr] \ar@{->}[rrrdd] & & & \ldots \\
        & \\
        & & & (1,1,1,1) \ar@{->}[rrr] & & & (2,1,1,1) \ar@{->}[rrr] \ar@{->}[rrrdd] & & & (3,1,1,1) \ar@{->}[rrr] \ar@{->}[rrrdd] & & & (4,1,1,1) \ar@{->}[rrr] \ar@{->}[rrrdd] & & & (2,1,1,1) \ar@{->}[rrr] \ar@{->}[rrrdd] & & & \ldots \\
        & \\
        & & & & & & & & & (3,2,1,1) \ar@{->}[rrr] & & & (4,2,1,1) \ar@{->}[rrr] \ar@{->}[rrrdd] & & & (5,2,1,1) \ar@{->}[rrr] \ar@{->}[rrrdd] \ar@{->}[rrrdd] & & & \ldots \\
        & \\
        & & & & & & & & & & & & & & & \ddots & & & \ddots
      }
\]
\end{center}
\vspace{-0.1cm}
\end{table}

\begin{table}[h]
\caption{The poset $\big(\mathcal T(I,2), \leqslant^I \big)$ for $G = \mathrm{SO}(9)$, $I = \{\alpha_2, \alpha_3 \}$.}
\begin{center}
\vspace{-0.6cm}
\[
\xymatrix@R=3pt@C=1pt{
        (0,0,1,1) \ar@{->}[rrr] & & & (0,1,1,1) \ar@{->}[rrr] \ar@{->}[rrrdd] & & & (0,2,1,1) \ar@{->}[rrr] \ar@{->}[rrrdd] & & & (0,3,1,1) \ar@{->}[rrr] \ar@{->}[rrrdd] & & & (0,4,1,1) \ar@{->}[rrr] \ar@{->}[rrrdd] & & & (0,5,1,1) \ar@{->}[rrr] \ar@{->}[rrrdd] & & & \ldots \\
        & \\
        & & & & & & (1,2,1,1) \ar@{->}[rrr] & & & (1,3,1,1) \ar@{->}[rrr] \ar@{->}[rrrdd] & & & (1,4,1,1) \ar@{->}[rrr] \ar@{->}[rrrdd] & & & (1,2,1,1) \ar@{->}[rrr] \ar@{->}[rrrdd] & & & \ldots \\
        & \\
        & & & & & & & & & & & & (2,4,1,1) \ar@{->}[rrr] & & & (2,5,1,1) \ar@{->}[rrr] \ar@{->}[rrrdd] \ar@{->}[rrrdd] & & & \ldots \\
        & \\
        & & & & & & & & & & & & & & & & & & \ddots
      }
\]
\end{center}
\vspace{-0.8cm}
\end{table}

\begin{table}[h]
\caption{The poset $\big(\mathcal T(I,2), \leqslant^I \big)$ for $G = \mathrm{SO}(9)$, $I = \{\alpha_1, \alpha_2 \}$.}
\begin{center}
\vspace{-0.4cm}
\[
\xymatrix@R=3pt@C=1pt{
		& & & & & (0,3,3,3) \ar@{->}[dd] & & \\
        (0,1,1,1) \ar@{->}[rrr] \ar@{->}[dd] & & & (0,2,1,1) \ar@{-}[rr]\ar@{->}[dd] \ar@{->}[ur] & & \ar@{->}[rr] & & (0,3,1,1) \ar@{->}[rrr] \ar@{->}[dd] & & & (0,4,1,1) \ar@{->}[dd] \ar@{->}[rr] & & \ldots \\
        & & & & & (1,3,3,3) \ar@{->}[dd] \\
        (1,1,1,1) \ar@{->}[rrr] \ar@{->}[dd] & & & (1,2,1,1) \ar@{-}[rr]\ar@{->}[dd] \ar@{->}[ur] & & \ar@{->}[rr] & & (1,3,1,1) \ar@{->}[rrr] \ar@{->}[dd] & & & (1,4,1,1) \ar@{->}[dd] \ar@{->}[rr] & & \ldots \\
        & & & & & (2,3,3,3) \ar@{->}[dd] \\
        (2,1,1,1) \ar@{->}[rrr] \ar@{->}[dd] & & & (2,2,1,1) \ar@{-}[rr]\ar@{->}[dd] \ar@{->}[ur] & & \ar@{->}[rr] & & (2,3,1,1) \ar@{->}[rrr] \ar@{->}[dd] & & & (2,4,1,1) \ar@{->}[dd] \ar@{->}[rr] & & \ldots \\
        & & & & & \vdots \\
\vdots & & & \vdots & & & & \vdots & & & \vdots & & &
    }
\]
\end{center}
\vspace{-0.4cm}
\end{table}

}

{\scriptsize

\begin{table}[h]
\caption{The poset $\big(\mathcal T(I,2), \leqslant^I \big)$ for $G = \mathrm{SO}(9)$, $I = \{\alpha_2 \}$.}
\begin{center}      
\vspace{-0.6cm}
\[
\xymatrix@R=3pt@C=1pt{
        & (0,1,1,1) \ar@{->}[rrr] \ar@{->}[rrrdd] & & & (0,2,1,1) \ar@{->}[ddllll] \ar@{->}[rrr] \ar@{->}[rrrdd] & & & (0,3,1,1) \ar@{->}[rrr] \ar@{->}[rrrdd] & & & (0,4,1,1) \ar@{->}[rrr] \ar@{->}[rrrdd] & & & (0,5,1,1) \ar@{->}[rrr] \ar@{->}[rrrdd] & & & \ldots \\
        & \\
       (0,3,3,3) & & & & (1,2,1,1) \ar@{->}[ddllll] \ar@{->}[rrr] & & & (1,3,1,1) \ar@{->}[rrr] \ar@{->}[rrrdd] & & & (1,4,1,1) \ar@{->}[rrr] \ar@{->}[rrrdd] & & & (1,2,1,1) \ar@{->}[rrr] \ar@{->}[rrrdd] & & & \ldots \\
        & \\
        (1,3,3,3) & & & & & & & & & & (2,4,1,1) \ar@{->}[rrr] & & & (2,5,1,1) \ar@{->}[rrr] \ar@{->}[rrrdd] & & & \ldots \\
        & \\
        & & & & & & & & & & & & & & & & \ddots 
      } 
\]
\end{center}
\vspace{-0.4cm}
\end{table}

\begin{table}[h]
\caption{The poset $\big(\mathcal T(I,2), \leqslant^I \big)$ for $G = \mathrm{SO}(9)$, $I = \{\alpha_1\}$.}
\begin{center}
\vspace{-0.4cm}
\[
\xymatrix@R=3pt@C=1pt{
        (1,1,1,1) \ar@{->}[dd] & & & & & & & & (3,3,3,3) \ar@{->}[dd] & & & & & & & & (5,5,5,5) \\
        & & & & & & & & & & & & \\
        (2,1,1,1) \ar@{->}[rrrrrdd] \ar@{->}[dd] \ar@{->}[uurrrrrrr] & & & & & & & & (4,3,3,3) \ar@{->}[dd] \ar@{->}[uurrrrrrr] \\
        & & & & & & & & & & & & \\
        (3,1,1,1) \ar@{->}[rrrrrdd] \ar@{->}[dd] & & & & & (3,2,1,1) \ar@{->}[dd] \ar@{->}[uurrr] & & & (5,3,3,3) \ar@{->}[dd] & & & \\
        & & & & & & & & & & & & \\
        (4,1,1,1) \ar@{->}[rrrrrdd] \ar@{->}[dd] & & & & & (4,2,1,1) \ar@{->}[dd] \ar@{-}[rrrd] \ar@{->}[uurrr] & & & (6,3,3,3) \ar@{->}[dd] & & & \\
        & & & & & & & & \ar@{->}[rrrd] \\
        (5,1,1,1) \ar@{->}[rrrrrdd] \ar@{->}[dd] & & & & & (5,2,1,1) \ar@{->}[dd] \ar@{-}[rrrd] \ar@{->}[uurrr] & & & (7,3,3,3) \ar@{->}[dd] & & & (5,3,1,1) \ar@{->}[dd] \\
        & & & & & & & & \ar@{->}[rrrd] \\
        (6,1,1,1) \ar@{->}[rrrrrdd] \ar@{->}[dd] & & & & & (6,2,1,1) \ar@{->}[dd] \ar@{-}[rrrd] \ar@{->}[uurrr] & & & (8,3,3,3) \ar@{->}[dd] & & & (6,3,1,1) \ar@{->}[dd] \ar@{->}[rrrrrdd] \\
        & & & & & & & & \ar@{->}[rrrd] \\
		\vdots & & & & & \vdots \ar@{->}[uurrr] & & & \vdots & & & \vdots & & & & & \ddots
				}
\]
\end{center}
\end{table}
}


\normalsize

\begin{thebibliography}{BGMR}

\bibitem{AB} V.~Alexeev and M.~Brion,
{\em Stable reductive varieties II: Projective case},
Adv. Math. \textbf{184} (2004), no. 2, 380--408.

\bibitem{Bo} N.~Bourbaki,
{\em \'El\'ements de math\'ematique. Fasc.\ XXXIV. Groupes et alg\`ebres de Lie.  Chapitres IV, V, VI}, Actua\-lit\'es Scientifiques et Industrielles \textbf{1337}, Hermann Paris 1968.

\bibitem{BGMR}
P.~Bravi, J.~Gandini, A.~Maffei and A.~Ruzzi,
\newblock{\em Normality and non-normality of group compactifications in simple projective spaces},
\newblock Ann. Inst. Fourier (Grenoble) \textbf{61} (2011), no. 6, 2435--2461.

\bibitem{DC}
C.~De Concini,
{\em Normality and non normality of certain semigroups and orbit closures},
in: Algebraic transformation groups and algebraic varieties,
Encycl. Math. Sci. \textbf{132}, Springer, Berlin, 2004, pp. 15--35.

\bibitem{CP}
C.~De Concini and C.~Procesi,
\newblock {\em Complete symmetric varieties},
\newblock in: Invariant theory (Montecatini, 1982), Lecture Notes in Math. \textbf{996},  Springer, Berlin, 1983, pp. 1--44.

\bibitem{GR}
J.~Gandini and A.~Ruzzi,
\newblock{\em Normality and smoothness of simple linear group compactifications},
\newblock Math. Z., to appear.

\bibitem{GW}
R.~Goodman and N.R.~Wallach,
\newblock {\em Symmetry, representations, and invariants},
\newblock Graduate Texts in Mathematics \textbf{255}, Springer, Dordrecht, 2009. Appendices E,F,G available at http://www.math.rutgers.edu/~goodman/repbook.html

\bibitem{Ka}
S.~S.~Kannan,
\newblock {\em Projective normality of the wonderful compactification of semisimple adjoint groups},
\newblock Math. Z. \textbf{239} (2002), no. 4, 673--682.

\bibitem{Kn} F.~Knop,
{\em The Luna-Vust theory of spherical embeddings},
\newblock in: Proceedings of the Hyderabad Conference on Algebraic Groups (Hyderabad, 1989), Manoj Prakashan, Madras, 1991, pp. 225--249.

\bibitem{KKLV} F.~Knop, H.~Kraft, D.~Luna and T.~Vust,
{\em Local properties of algebraic group actions},
\newblock in: Algebraische Transformationsgruppen und Invariantentheorie, DMV Sem. \textbf{13},
Birkh\"auser, Basel, 1989, pp. 63--75.

\bibitem{Ku}
S.~Kumar,
\newblock {\em Tensor Product Decomposition},
\newblock in: Proceedings of the International Congress of Mathematicians, Hyderabad, India, 2010, pp. 1226--1261.

\bibitem{Ti}
D.~A.~Timashev,
\newblock {\em Equivariant compactifications of reductive groups},
\newblock  Sb. Math.  \textbf{194} (2003), no. 3--4, 589--616.

\end{thebibliography}
\end{document}